\documentclass%[11pt]
{amsart}
\usepackage{amssymb,amsthm,amsmath,amsfonts,euscript,qsymbols,delarray}
\usepackage{latexsym,amstext,amscd}
\usepackage{enumerate,hyperref}
\usepackage{color}

\newtheorem{theorem}{Theorem}[section]

\newtheorem{proposition}[theorem]{Proposition}
\newtheorem{corollary}[theorem]{Corollary}
\newtheorem{lemma}[theorem]{Lemma}

\newtheorem{introtheorem}{Theorem}
\newtheorem{introcorollary}{Corollary}
\newtheorem{claim}
{Claim}
\newtheorem*{claim*}{Claim}
\theoremstyle{definition}
\newtheorem{example}[theorem]{Example}
\newtheorem{remark}[theorem]{Remark}
\newtheorem{definition}[theorem]{Definition}
\newtheorem{fact}[theorem]{Fact}

\newtheorem{question}[theorem]{Question}

\newcommand{\C}{\mathfrak{c}}
\newcommand{\T}{\mathbb{T}}
\newcommand{\Z}{\mathbb{Z}}
\newcommand{\N}{\mathbb{N}}
\def\hull#1{\langle#1\rangle}

\usepackage{vmargin}
\setmarginsrb{14mm}{6mm}{14mm}{9mm}%
{14mm}{7mm}{14mm}{9mm}
\catcode`\@=12

\def\T{{\mathbb T}}

\def\Z{{\mathbb Z}}
\def\N{{\mathbb N}}
\def\R{{\mathbb R}}
\def\Q{{\mathbb Q}}
\def\P{{\mathbb P}}
\def\C{{\mathbb C}}

\def\K{{\mathbb K}}

\def\Prm{\P}

\def\Soc{\mathrm{Soc}}
\def\nbd{neighbourhood}
\newcommand{\cc}{countably\ compact}
\newcommand{\mi}{minimal}

\newcommand{\lm}{locally minimal}
\newcommand{\co}{connected}
\newcommand{\sco}{sequentially complete}

\def\KK{\mathcal K}
\def\CCC{\mathcal C_{Clm}}
\begin{document}

\title[Sequentially complete locally minimal groups]{The Structure of Sequentially Complete Locally Minimal Groups}

\author[D. Dikranjan, W. He and D. Peng]{Dikran Dikranjan, Wei He and Dekui Peng}


%%%%%%%%%%%%%%
\date{}
%{\today}

\begin{abstract} 
Generalizing results from \cite{DTk,DU} we study the fine structure of locally minimal (locally) precompact Abelian groups 
(these are the locally essential subgroups $G$ of LCA groups $L$, i.e., such that 
$G$ non-trivially meets all ``small" closed subgroup of $L$). More precisely we prove that if $G$ is a dense locally minimal and 
sequentially closed subgroup of a LCA group $L$, then the connected component $c(G)$ of $G$ has the same weight as $c(L)$.
Moreover, when $w(c(G))$ is not Ulam measurable, then $c(G) = c(L)$. We provide an extended discussion  illustrating 
how this result fails in various ways in the non-abelian case (even for nilpotent groups of class 2).

Motivated by the above result, we study further those locally minimal precompact Abelian groups $G$, termed {\em critical locally minimal},
such that $c(G) =c(K)$ (where $K$ is the compact completion of $G$) and 
%A particular attention is paid to a new class of locally minimal precompact Abelian groups $G$, termed {\em critical locally minimal}. Namely, those  locally minimal groups 
$G/c(G)$ is not locally minimal.
% (see Definition \ref{Def:crit:loc:min}). 
%Even if every minimal group is always locally minimal, minimal  (and in particular, 
Such a group cannot be compact, neither connected, nor totally disconnected. 
%ompact) Abelian groups are {\em never} critical locally minimal (even if they are always precompact).  
We provide a proper class of critical locally minimal 
%Abelian 
groups with additional compactness-like properties  and we study the class $\CCC$ of compact Abelian groups with a dense critical locally minimal subgroup. In particular, we completely describe the connected components of the finite-dimensional 
%compact Abelian 
groups belonging to $\CCC$. 
\end{abstract}

\thanks{$^\ast$The second and third authors acknowledge the support of NSFC grant Nos.~12271258 and 12301089.}
\maketitle

%\tableofcontents
%

\section{Introduction} 

\subsection*{Notation and terminology} 
We denote by ${\N}$ and ${\Prm}$ the sets of naturals ($\{0,1,\dots\}$) and primes, 
respectively,  by ${\Z}$ the integers, by ${\Q}$  the rationals, by  ${\R}$ the reals,  by ${\T}$  the unit circle
group in ${\C}$,  by $\Z_p$ the $p$-adic integers ($p\in \Prm$), by $\Z(n)$ the cyclic group of
order $n$ ($n\in \N$). The cardinality of continuum $2^\omega$ will be denoted also by ${\mathfrak c}$.

A cardinal $|X|$ is {\it Ulam-measurable} if there exists a free ultrafilter on $X$ stable under countable intersections. The assumption that there exist no Ulam-measurable cardinals is  known to be consistent with ZFC, while their existence implies consistency of ZFC and hence no proof is to be expected \cite{J}. 

% An ultrafilter on a set $X$ is $\kappa$-complete if it is closed under intersections of families of cardinality $<\kappa$. A cardinal $|X|$ is {\it Ulam-measurable} 
%(resp., {\it measurable}) if there exists an $\omega_1$-complete (resp., $|X|$-complete) free ultrafilter on $X$. Uncountable measurable cardinals are Ulam-measurable, while the least Ulam-measurable cardinal is also measurable.  A cardinal is Ulam-measurable if and only if it greater than or equal to the first uncountable measurable cardinal.  As usual \cite[Definition 10.3]{J}, an uncountable cardinal  $\kappa$ is measurable if there exists a  $\kappa$-complete free ultrafilter on  $\kappa$.
%%  Let us recall that a cardinal $\alpha$ is {\it Ulam-measurable} (resp., {\it measurable}) if there exists an ultrafilter on $\alpha$ which is closed under countable intersections (resp., under intersections of $<\alpha$ sets). 

All groups considered in this paper, unless otherwise stated (e.g., in \S 6), will be Abelian, so additive notation will be always used. Let $G$  be a group and  $A$  be a subset of $G$. We denote  by $\langle A\rangle$ the subgroup of $G$ generated by $A$. The group $G$ is {\it divisible}  if for every $g\in G$ and positive $n\in \N$ the equation $nx=g$ has a solution in $G$. For an  Abelian group $G$ we  set $G[n] = \{x \in  G: nx=0\}$, for $n\in \N$. We set $\Soc(G) = \bigoplus_{p\in \Prm}G[p]$, and $t(G)$ the subgroup consisting of all torsion elements; while 
$r(G)$ will denote the free rank of $G$ and $r_p(G)$ -- its $p$ rank, when $p\in \P$.

All topological groups considered in this paper are Hausdorff, so Tychonoff spaces. We denote by $w(G)$ the weight of a topological group $G$, so if
$G$ is finite one has $w(G) = |G|$.  The Pontryagin dual of a topological group $G$ will be denoted by $\widehat G$, while  $\widetilde G$
will denote its completion and 
%(we do not have to worry about left, right, and two-sided uniformities, since our groups are Abelian),
%Completeness of topological groups is intended with respect to the two-sided uniformity, 
%so that  every topological group $G$ has a (Ra\u \i kov) completion which we denote by $\widetilde G$, 
$c(G)$ denotes the \co \ component of the neutral element of a group $G$ (for brevity, we call $c(G)$ the connected component of $G$). 
Moreover, $G$ is said to be \sco\  if every Cauchy sequence  in $G$ converges (equivalently, $G$ is sequentially closed in $\widetilde G$). 
We denote by ${\mathbb K}$ the (compact)  Pontryagin dual of the discrete group $\Q$. A topological group is said to be \emph{monothetic} if it contains a dense cyclic subgroup. In particular, every monothetic group is Abelian. Moreover, a monothetic group is called \emph{$p$-monothetic} if it is either a cyclic $p$-group or is topologically isomorphic to a subgroup of $\Z_p$.
% i.e., ${\mathbb K}$ is the group of all homomorphisms $\Q\to \T$ equipped with the topology of pointwise convergence. 

Recall that a
%some notions of compactness-like conditions in topological groups and spaces. A 
Tychonoff space $X$ is {\it pseudocompact} if every continuous real-valued function on $G$ is bounded, and 
{\it countably compact} if every countable open cover has a finite subcover.  A topological group  $G$  is {\it (locally) precompact} if its  completion $\widetilde G$ is (locally) compact. 
%(or, equivalently, if for any open $U\ne \emptyset$ in $G$   there is a finite subset $F\sq G$ such that $F+U=G$).
For an infinite cardinal $\alpha$, a group $G$ is {\it $\alpha$-bounded } if every subset of cardinality  $\leq\alpha$  of $G$  is contained in a compact subgroup of $G$.
So a group  $G$ is $\omega$-bounded precisely when all closed separable subgroups of $G$ are compact. 
% We say that $G$ is  {\it $\omega$-bounded} if every countable subset is contained in a compact subgroup (). 
 Then $\omega$-boundedness implies countable compactness, and countable compactness  implies precompactness. 
%More generally, f 
%$\alpha$-bounded always implies initially $\alpha$-compact and initially $\omega$-compact coincides with \cc. 
  
For undefined symbols or notions see \cite{DPS,Eng,Fuchs}.

 \subsection{Historical background on minimal groups and their structure}\hfill
 
 %%%%%%%%%%%%%%%%%%%%%%%%%%%%%%%%%%%
% A topological group $G$ is called {\it (locally) minimal\/} if (there exists a neighbourhood $V$ of the identity of $G$ such that) whenever  $H$ is a Hausdorff group and $f:G\to H$ is a continuous isomorphism (such that $f(V)$ is a neighbourhood of the identity in $H$), then $f$ is open. 
%%%%%%%%%%%%%%%%
 A topological group $G$ is called {\it minimal\/} if whenever  $H$ is a Hausdorff group and $f:G\to H$ is a continuous isomorphism, then $f$ is open. 
 Minimal groups were introduced by Choquet at the 1970 ICM in Nice, as an obvious generalization of compact groups.
  The first examples of non-compact minimal groups were found by Do\"\i chinov \cite{Do1} and Stephenson \cite{St}. 
A topological group $G$ is called {\em totally minimal}, if all its Hausdorff quotients are minimal \cite{DP1}.  
Clearly, these are the groups satisfying the Open Mapping Theorem. The research in this field was quite intensive for more than five decades (for a quite incomplete list see
\cite{B,DS_PAMS,DS3,Meg,P1,PS,RS}, 
%Sw1,Sw2}, 
as well as the surveys or monographs \cite{D,D:omt,DMeg,DPS}). 
 
 In some cases, minimal groups may have some additional compactness property, e.g., 
  every minimal Abelian group is precompact according to the celebrated Prodanov-Stoyanov Theorem \cite{PS}. This motivates the study of the \mi \ groups satisfying additional compactness properties (like pseudocompactness, sequential completeness, countable compactness, $\omega$-boundedness, etc.).
   This trend proved to  
be very fruitful and has been carried out by many authors (\cite{DS_PAMS}, \cite{DS3}, \cite{DTk}, \cite{DT}, \cite{DU}, \cite{RS}). 
For example, the structure of \cc\ \mi\ groups was one of the main topics of the surveys \cite[\S 5]{D}, \cite[\S 2.4]{D-ConComp} and \cite[\S 4]{D:omt}. 
Here ``structure" has to be understood as a reduction to the compact case (i.e., showing that the minimal groups with those additional properties
is either compact, or have a large portion that is compact), e.g., totally minimal \cc \ Abelian groups were shown to be compact in \cite{DS_PAMS} and this fact was extended to all nilpotent groups in \cite[Theorem 5.3]{D}. On the other hand, only minimality and countable compactness do not give such a strong immediate impact (\cite{DS3} contains some results related to the minimality of products of such groups).
 The above mentioned surveys reported also some deeper results on the structure of the  \cc\ \mi\ groups from the unpublished manuscript  \cite{D-cc}, more specifically the following two regarding connected groups: 

(i) \cite{D-cc}  the countably compact \mi\ connected Abelian groups of non-Ulam-measurable weight are compact;  

(ii)  \cite{D-cc} for every Ulam-measurable cardinal $\kappa$ there exists a non-compact $\omega$-bounded \mi\ connected Abelian group of weight $\kappa$. 

The fact (covering (i) and (ii)) that the Ulam-measurable cardinals are precisely the weights of \cc \ 
%(or even $\omega$-bounded) 
\mi\ connected non-compact Abelian groups was quoted in \cite[Theorem 5.7]{D}. In \cite[Theorem 4.3]{D:omt} ``countable compactness" was significantly weakened to ``sequential completeness",
% (see Definition \ref{SecCo}), 
 a proof of (ii) appeared
 only 
 recently in \cite{DU}.

(i) was announced without proof in \cite[Theorem 5.10]{D}, \cite[Theorem 4.6]{D:omt} and  \cite[Theorem 2.4.2]{D-ConComp}. In \cite[Corollary 2.4.3]{D-ConComp} and  \cite[Th. 5.12(b)]{D} the following stronger result is deduced from (i): if $G$ is a countably compact \mi\ connected Abelian group and $w(c(G))$ is not Ulam-measurable, then $c(G)$ is compact and $G/c(G)$ is minimal. 
%(\cite[Theorem 2.4.2]{D-ConComp}). 
Combining this fact with Fact \ref{Fact:June6}, one obtains also the equality $c(G) = c(\widetilde G)$ when $w(c(G))$ is not Ulam-measurable. The proof of (i) appeared much later in \cite[Theorem 3.2]{DTk} in the following sharper form: 

\begin{theorem}\label{Thm:DTk} {\rm \cite{DTk}} If $G$ is a minimal sequentially complete Abelian group, then  $w(c(G)) = w(c(\widetilde G))$. If $w(c(G))$ is not Ulam-measurable, then $c(G) = c(\widetilde G)$ is compact and $G/c(G)$ is still minimal.
\end{theorem}

The aim of this paper is to extend this theorem to a larger class of Abelian groups specified in \S \ref{Sec:MR}.
%study local minimality accompanied by some other compactness property, as sequential completeness, in order to extend the above results to the larger class of locally minimal \sco \ groups that are also locally precompact.

\subsection{Main results}\label{Sec:MR}
\hfill 

%The aim of this paper is the
To find a correct counterpart of Theorem \ref{Thm:DTk} for a larger class of groups naturally embracing the minimal ones we make recourse
to a class introduced by Morris and Pestov \cite{MP} with the motivation that the locally compact groups need not be minimal 
(as a locally compact Abelian group is minimal precisely when it is compact \cite{St}). They (and independently, Banakh \cite{Ban})  introduced 
the notion of locally minimal group as follows: a topological group $G$ is {\it locally minimal\/} if there exists a neighbourhood $V$ of the identity of $G$ such that whenever  $H$ is a Hausdorff group and $f:G\to H$ is a continuous isomorphism such that $f(V)$ is a neighbourhood of the identity in $H$, then $f$ is open. 
% (sometimes, for simplicity, we simply speak of local minimality without mentioning an explicit \nbd\ $V$). 
%If we want to point out that the neighbourhood $V$ witness  of $(G, \tau)$, we say that $(G, \tau)$ is \emph{$V$-locally minimal} or $(G, \tau)$ is \emph{locally %minimal with respect to $V$.}
It turned out that  locally compact groups (and  normed vector spaces) are locally minimal \cite{MP}. Further details on locally minimal groups can be found in \cite{DM,ACDD1,ACDD2}. 
The relevant permanence properties of local minimality related to the passage to closed or dense subgroups were largely studied in these papers
as well as in \cite{DHXX,XDST,XDST1}. In particular, local minimality is preserved by the passage to dense  locally essential
subgroups (see also \S \ref{SecBack} and Definition \ref{Def:crit:loc:min} for local essentiality). 

 We prove that the first part of Theorem \ref{Thm:DTk} (up to and including ``$c(G) = c(\widetilde G)$") holds even for locally minimal locally precompact \sco \ Abelian groups.  
 
% That is why  % (e.g., normed vector spaces are locally minimal). now (local) precompactness needs to be imposed explicitly when dealing with locally minimal Abelian groups. since locally minimal Abelian groups need not be precompact, unlike the minimal ones.

\medskip

\begin{introtheorem}\label{IntroA}
Let $G$ be a locally precompact, sequentially complete, locally minimal  Abelian group. Then 
\begin{itemize}
      \item[(a)] $w(c(G))=w(c(\widetilde{G}))$.
      \item[(b)] if $w(c(G))$ is not Ulam-measurable, then $c(G)=c(\widetilde{G})$.
      %  in particular, $c(G)$ is locally compact. 
\end{itemize}
\end{introtheorem}

This theorem 
(see \S \ref{Sec:LocMin:vs:Min} for the proof) extends Theorem \ref{Thm:DTk} not only in the direction of minimality
(which implies precompactness due to Prodanov-Stoyanov Theorem). Here the (implicit) precompactness in Theorem \ref{Thm:DTk} is weakened to 
local precompactness (which may be missing in the case of local minimality in general). 

%Here (local) precompactness needs to be imposed explicitly since the Prodanov-Stoyanov Theorem  does not extend to local minimality, in fact a locally minimal Abelian group need not be even locally precompact. 
 
%. It gives the following 

\begin{introcorollary}\label{CoroA} Let $G$ be a (locally) precompact, sequentially complete, locally minimal Abelian group. 

(a) If $w(c(G))$ is not Ulam-measurable, then $c(G)=c(\widetilde{G})$ is (locally) compact. 
 
 (b) If $G$ is hereditarily disconnected $($i.e., $c(G) = \{0\})$, then so is $\widetilde G$; in particular, $\widetilde G$ and $G$
 have linear topology.
% 
% \begin{corollary}\label{Coro:Oct23} 
% If $G$ is a locally minimal, hereditarily disconnected  \sco \ and  locally precompact Abelian group, then $c(\widetilde G) = 0$
% $($or equivalently, $\dim \widetilde G = 0)$. Consequently, $G$ has 
% 
\end{introcorollary}

Here {\em linear topology} means that the group has a local base of open subgroups (see Corollary \ref{Coro:Oct23} for a proof of (b)). % \NB
% If $G$ is a locally minimal, hereditarily disconnected  \sco \ and  locally precompact Abelian group, then $c(\widetilde G) = 0$
% $($or equivalently, $\dim \widetilde G = 0)$. Consequently, $G$ has linear topology.

%The counterpart of the final part of Theorem \ref{Thm:DTk} ($G/c(G)$ is minimal) will be discussed below. 
%in the next \S \ref{Clm} % \subsection{Critical  locally minimal groups}\label{Clm}\hfill 
 
The counterpart for local minimality of  the missing (final) part of Theorem \ref{Thm:DTk} (namely, ``$G/c(G)$ is locally minimal") spectacularly fails 
 even for {\em precompact} (locally minimal,  \sco)\ Abelian groups. The second part of the paper (namely, \S \ref{Sec:Clm}) is entirely dedicated to clarifying in full detail this
 remarkable  phenomenon. 

Quotients of locally minimal groups need not be locally minimal in general  (see \cite{DDHXX} for more details). In order to better pursue our plan we introduce a class of seemingly very exotic locally minimal groups, where the desired property fails in a strong way. 
Namely, we are interested only in the preservation of local minimality under very special quotients of a given a locally minimal group $G$, namely those of the form $G/c(G)$ when $c(G)$ is compact! As we prove in Corollary \ref{CoroA}, if the  locally minimal group $G$ is additionally precompact, \sco \ and $w(c(G))$ is not Ulam-measurable, then $c(G) = c(\widetilde G)$ is compact.  So it is natural to impose compactness of $c(G)$ in the following notion:
%  which isolates those locally minimal precompact group $G$ for which $c(G)=c(\widetilde{G})$ is compact and the quotient $G/c(G)$ fails to be locally minimal: 

\begin{definition} Call a locally minimal group $G$ {\em critical}, if $c(G)= c(\widetilde G)$ is compact and $G/c(G)$ is not locally minimal. 
\end{definition}

Obviously, a critical locally minimal group can be neither connected nor totally disconnected, so $\{0\} \ne c(G) \ne G$. Moreover, critical locally minimal Abelian groups  cannot be minimal either (indeed, if $G$ is a minimal, then $G/c(G)$ is even minimal (see Corollary \ref{Coro:MinG/c(G)}).
This explains why critical locally minimal Abelian groups are quite hard to come by even beyond the class of precompact locally minimal groups
 (e.g., a large class of non-precompact locally minimal Abelian groups is provided by the additive groups of normed spaces $V$, which obviously cannot be critical
 as $c(V)=V$). Theorem \ref{IntroB} produces a surprizingly large class of critical precompact locally minimal Abelian groups, some of them with strikingly strong compact-like properties (e.g., $\alpha$-boundedness, with arbitrarily large $\alpha$).

In this direction we study the completions of the critical \lm \ precompact Abelian groups, in other words, the class $\CCC$ of compact Abelian group with a dense (necessarily proper) critical locally minimal subgroup.  In order to provide sufficient conditions for $K \in \CCC$ we study also the  class $\KK$ of compact Abelian groups $K$ with locally essential connected component $c(K)$. Then we provide in Theorem \ref{Thm:New}  a very simple characterization of the groups 
 from $\KK\cap \CCC$ used in the proof of Theorem \ref{IntroB}. 

We show that $K \not \in \CCC$ when $c(K)$ is a hybrid torus (see Definition \ref{hybridtori}) or when it splits as a topological direct summand (see Theorem \ref{coro:tor_free**} for other sufficient conditions for $K \not \in \CCC$ ). Combined with other intermediate results this gives a complete description of the connected components of the finite-dimensional compact Abelian groups $K\in \CCC$: 

%\begin{theorem}\label{MainThm*}
%\noindent{\bf Theorem B.}
\begin{introtheorem}\label{IntroB}  Let $C$ be a finite-dimensional connected compact Abelian group. Then the following are equivalent:
\begin{itemize}
\item[(a)] $C$ is a not a hybrid torus;
\item[(b)] there is a compact Abelian group $K\in \CCC$ with $c(K)\cong C$; moreover, the dense critical locally minimal subgroup $G$ of $K$ witnessing $K\in \CCC$
can be chosen to be either metrizable or $\omega$-bounded of weight $\kappa$, for any 
%arbitrarily assigned 
uncountable cardinal $\kappa$. 
\end{itemize}
\end{introtheorem}
%\end{theorem} \medskip

This theorem will be proved at the end of \S \ref{Sec:Clm}. It not only characterizes the connected components of the critical \lm \ precompact Abelian groups, but
also shows that,  for every arbitrarily chosen infinite cardinal $\alpha$, such a (necessarily, connected) group can be also the connected component of a critical \lm \ $\alpha$-bounded Abelian group. 

This theorem provides a large class of finite-dimensional groups $K \in \CCC$ (in particular all those with torsion-free $c(K)$, see Corollary \ref{exa:Thm:qt*}). In particular,
we show that there is a proper class of pairwise non-homeomorphic groups $K\in \CCC$ with some further properties of their connected components 
%are $\mathfrak c$ many pairwise non-homotopically equivalent, 
(see Remark \ref{Homotop}). 

%As we show in  such a construction is NOT possible if $c(K)$ is a torus (while in Example \ref{exa:Thm:qt}(a)  we carry out the construction when ). So the question can be from the point of view of $c(K)$ as follows (we simplify it to the extreme by asking   $\dim K = 1$):

% \subsection{Locally essential subgroups of the (locally) compact groups}\label{Clm} \hfill

Our choice to concentrate the paper on the Abelian case is motivated by the fact that all of our principal results (as well as Theorem \ref{Thm:DTk}) fail in
the non-Abelian case (relevant examples are provided in the final section 6) which was not discussed in \cite{DTk,DU}.

%\NB 
Section \ref{background} provides 
%in \S \ref{Sec:delta} 
background on general properties of locally minimal groups. It is focused on one of the main technical tools for the proof of Theorems \ref{IntroA} and \ref{IntroB} , namely the notion of (locally) essential subgroup along with the new notion of weakly essential subgroup 
(see \S \ref{SecBack} and Definition \ref{Def:crit:loc:min}). The relevance of these notions comes from the crucial fact that the locally essential subgroups of the (locally) compact groups are precisely the (locally) precompact locally minimal groups (this follows from the criterion  \ref{Crit} for (local) minimality of dense subgroups). In particular, every subgroup of a Lie group is locally essential, hence locally minimal.
% the locally essential dense subgroups of the locally compact groups have a relevant property related to the open mapping theorem -- they are locally minimal. . . .  
To enhance and understand better the local essentiality we make use of distinguished subgroups of the compact groups 
(as co-NSS  subgroups and $\delta$-subgroups, see Definition \ref{Def:D}) which are used also to introduce in \S \ref{Sec:delta}  some classes of compact Abelian groups  
which are frequently used in the proofs (Archimedean, Non-Archimedean, hybrid and exotic tori). In relation with critical locally minimal groups, \S \ref{Special:quot} is focused on hereditarily disconnected quotients of (locally) minimal groups.

%The paper is organized as follows.  as well as various properties of the compact Abelian groups.  In \S \ref{SecBack} we recall among others the  
 %%%%%%%%%%%%%%
%  In \S \ref{Sec:delta} various properties of the compact Abelian groups and their subgroups   (e.g., their $\delta$-subgroups, NSS- and co-NSS subgroups)
%%  ) In \S \ref{Ht}  and co-NSS subgroup of the compact Abelian groups 
%%  6 Tori: Archimedean, Non-Archimedean, hybrid and exotic tori
%  which allows us to focus on some special classes of compact Abelian groups as: Archimedean, Non-Archimedean, hybrid and exotic tori which are closely related to the topic. 
%  Finally, \S 2.4 is focused on hereditarily disconnected quotients of (locally) minimal groups. 

In \S \ref{Subsec:approx}  we discuss various ways of approximating locally minimal Abelian groups by (large) minimal
subgroups. Since minimal Abelian groups are precompact, this 
 shows that complete locally minimal Abelian groups contain compact $G_\delta$-subgroups (Corollary \ref{Coro2:May27}), 
 so in the non-metrizable case they are ``close to being" locally compact (see also Question \ref{QuesX1}). 
 %\S \ref{Sec:Conditions} 
 Theorem \ref{claim4}
provides necessary conditions and sufficient conditions for local essentiality of subgroups which are extensively used in the sequel. %\NB

In \S \ref{Sec4} we add sequential completeness
%some additional compactness conditions
 to local minimality and prove Theorem \ref{IntroA}, extending in this way 
 the first part of Theorem \ref{Thm:DTk} to  locally minimal locally precompact Abelian groups.   In this section we provide also a structure theorem for the hereditarily disconnected locally minimal locally precompact \sco\ Abelian groups. 
 
 In Section \ref{Sec:Clm} we show that the counterpart of the second part of Theorem \ref{Thm:DTk} strongly fails in many ways, for locally minimal groups. Here we prove Theorem \ref{IntroB}  which describes the connected component of a locally minimal locally precompact  \sco\ Abelian group. 
 
 % we build critical locally minimal Abelian groups that can be choosen to be  metrizable or $\omega$-bounded. 
 
 The final Section \ref{Sec:Final} is focused mainly on discussion of the non-Abelian case. Here we show that most of our principal results
 fail in this case, which motivated our choice to concentrate the paper on the Abelian case. 

\section{Preliminaries}\label{background}

%\subsection*{Terminology and notation:}

\subsection{General properties of  (locally) minimal groups}\label{SecBack}

%\subsection{Minimality and local minimality criteria}
\hfill

In the next definition we recall the well known concepts of (local) essentiality and total density of a subgroup. Moreover, in items (i) and (ii) we provide also weaker versions of essentiality. 

\begin{definition}\label{Def:crit:loc:min}
For a topological group $G$, a subgroup $H$ is said to be 
\begin{itemize}
  \item[(i)]  {\em (locally)} {\em essential} in $G$ if (there exists a \nbd \ $V$ of $e$ such that) $H$ non-trivially meets every non-trivial closed normal subgroup of $G$ (contained in $V$);
  \item[(ii)] {\em weakly essential} in $G$ if for every prime $p$,  $H$ non-trivially meets every  infinite closed $p$-monothetic subgroup of $G$;
  \item[(iii)] {\em totally dense} in $G$ if $H\cap N$ is dense in $N$ for any closed normal subgroup $N$ of $G$.
\end{itemize}
\end{definition}

In general it is hard to directly verify (local) essentiality since it involves {\em all} closed normal subgroups. A natural idea is to find a family $\mathcal{C}$ of non-trivial closed normal subgroups of $G$ such that every non-trivial closed normal subgroup of $G$ contains some element in $\mathcal{C}$. So to check the (local) essentiality or total density of $H$ in $G$, one can focus only $H\cap N$ for all $N\in \mathcal{C}$. This method is particularly useful when studying (total) minimality and local minimality of precompact Abelian groups.  
In the proof of Lemma \ref{Neeew}(b) we show that the family of all closed $p$-monothetic subgroups of $G$ when $p$ varies among all primes, can be taken for such a family $\mathcal{C}$, when $G$ is a compact Abelian group. For every prime $p$, the union $td_p(G)$ of the family of all $p$-monothetic subgroups of an Abelian group $G$ is a subgroup of $G$ \cite{S}.

%Therefore, when working with compact Abelian groups, one may take $\mathcal{C}$ to be the class consisting of compact groups topologically isomorphic to either $\Z(p)$ or $\Z_p$, as $p$ ranges over all primes.

\begin{lemma}\label{Neeew} Let $G$ be a compact Abelian group. Then for every $H\leq G$

\begin{itemize}
   \item[(a)] ``essential'' $\to $ ``locally essential" $\to $ ``weakly essential";  
   \item[(b)]  essential is equivalent to  weakly essential and $\Soc(G) \leq H$;
   \item[(c)] local essentiality of $H$ is equivalent to weak essentiality of $H$ and local essentiallity of $\Soc(H)$ in $\Soc(G)$.
\end{itemize}
\end{lemma}

\begin{proof} 
(a) To see that  local essentiality $\to $ weak essentiality, assume that $V$ witnesses local essentiality of $H$ in $G$. Then 
for every non-trivial closed $p$-monothetic subgroup $N$ of $G$ the intersection 
$N \cap V$ contains an open subgroup $p^nN \cong N$ of $N$, which is contained in $V$, so $p^nN \cap G \ne \{0\}$. Therefore, $N\cap G \geq p^nN \cap G \ne \{0\}$.
The remaining implications in (a) are obvious. 

(b) The necessity is obvious. The sufficiency immediately follows from the fact that every non-trivial compact Abelian group $N$ 
contains a closed subgroup topologically isomorphic to either $\Z(p)$ or $\Z_p$ for some prime $p$. Indeed, let $X = \widehat N \ne \{0\}$. If $X$ is divisible, then by the structure theorem for divisible Abelian groups \cite{Fuchs}, 
$X$ admits a quotient group isomorphic to $\Z(p^\infty)$ for some prime $p$. Then $N$ has subgroup topologically isomorphic to $\Z_p$. 
If $X$ is not divisible, then $pX \neq X$ for some prime $p$,  so  the quotient $X/pX$ is non-trivial, and hence $X/pX$, as well as $X$, 
admit a quotient group isomorphic to $\Z(p)$. Then $N$ has subgroup topologically isomorphic to $\Z(p)$. 

%Let $X$ be any infinite (discrete) Abelian group. If $X$ is not $p$-divisible for some prime $p$ (i.e., , then the quotient $X/pX$ is non-trivial, and hence $X$ admits a quotient group isomorphic to $\Z(p)$. On the other hand, if $X$ is $p$-divisible for all primes $p$—that is, if $X$ is divisible—then, 
%follows from the argument  preceding the lemma: if $H$ is weakly essential and contains $\Soc(G)$, then $H$ intersects non-trivially with every subgroup in the class $\mathcal{C}$, and hence is essential in $G$. 
 
(c) If $H$ is weakly essential in $G$ and $\Soc(H)$ is locally essential in $\Soc(G)$, then to check local essentiallity of $H$ in $G$ it suffices to fix a \nbd \ $U$ of 0 in $\Soc(G)$ witnessing local essentiallity of $\Soc(H)$ in $\Soc(G)$. Pick a \nbd \ $W$ of 0 in $G$ such that $W \cap \Soc(G) = U$. If $S \leq G$, $S \subseteq W$ and $S \cong \Z(p)$ for some prime $p$, then $S \leq \Soc(G)$ as well, so $S \subseteq U$. Hence, $S \leq \Soc(H)\leq H$, as desired. 
\end{proof}

The next lemma shows that weak essentiality is preserved under  inverse images and under some images: 

\begin{lemma}\label{Rem:29} Let $q: K \to K_1$ be a continuous homomorphism of compact Abelian groups.
\begin{itemize}
  \item[(a)] if $H$ is a weakly essential subgroup  of $K_1$, then $H_1:= f^{-1}(H)$ is a weakly essential subgroup  of $K$; 
  \item[(b)] if $q(K)$ is weakly essential in $K_1$, then  $q(G)$ is weakly essential in $K_1$ for any weakly essential subgroup $G$ of $K$.
\end{itemize}
\end{lemma}

\begin{proof} (a) If $L$ be closed infinite $p$-monothetic subgroup of $K$, then its image $f(L)$ is either finite and isomorphic to $\Z(p^n)$
for some $n\in \N$ (since all non-trivial closed subgroups of $L$ have the form $p^nL$), or isomorphic to $L$. 
%\NB\footnote{was:
%In the latter case $p^nL = \ker f \cap L$ is contained in $H_1$, so $L \cap H_1 \ne 0$. In the forme case 
%$f(L) \cong L$ is a  closed infinite $p$-monothetic subgroup of $K_1$, so 
%$f(L) \cap N \ne 0$. Pick a non zero element $f(x) \in f(L) \cap N$, with $x\in L$. Then obviously $x \in H_1$ as well witnessing again $L \cap H_1 \ne 0$. }
In the former case $p^nL = \ker f \cap L$ is contained in $H_1$, so $L \cap H_1 \ne \{0\}$.  
In the latter case  $f(L) \cong L$ is a  closed infinite $p$-monothetic subgroup of $K_1$, so $f(L) \cap N \ne \{0\}$.
 Pick a non zero element $f(x) \in f(L) \cap N$, with $x\in L$. 
 Then obviously $x \in H_1$ as well witnessing again $L \cap H_1 \ne \{0\}$. 

(b) % As $q$ factors as the composition of a surjective homomorphism $K \to q(K)$ and an injective homomorphism $q(K) \hookrightarrow K_1$, it is sufficient to consider the cases where $q$ is either surjective or injective.  
%
%The case when $q$ is injective is simply transitivity of  weakly essentiality. Therefore it only remains to check the case when $q$ is surjective. 
 Let $N$ be  closed infinite $p$-monothetic subgroup of $K_1$. 
%\NB\footnote{was: pick an element $y\in \widetilde G$ with $q(y) = x$.\\ There was a gap here.  This $y$ need not exist if $x\not \in cl(q(G)$. Since we do not assume density of $G$ in $K$,  we cannot claim that $q(y) \in cl(q(G))$ if one wants to have $y\in \widetilde G = cl_K(G)$, then necessarily 
%$q(y) \in cl(q(G)$. } 
By the weak essentiality of $q(K)$ in $K_1$, $N_1:= N \cap q(K)\ne \{0\}$ is a closed $p$-monothetic subgroup of $q(K)$, so $N_1 \leq td_p(q(K))$. Since $q: K\to q(K)$ is surjective, $td_p(K) = td_p(q(K))$ (\cite{DLLM,DPS}). Pick a topological generator $x$ of $N_1$, so $x\in td_p(q(K))$. Then there exists $y\in td_p(K)$ with $q(y) = x$. Then for closed subgroup $N_2$ of $K$ generated by $y$, $q\restriction_{N_2}:N_2 \to N_1$ is an isomorphism. Since $N_2 \cap G\ne \{0\}$, it follows that $N_1\cap q(G)\ne \{0\}$ as $q\restriction_{N_1}$ is an isomorphism.
%
% Consider closed subgroup $N_2$ of $K$ generated by $y$. Since $N = q(N_2)$ is isomorphic to a quotient of $N_2$, and since the Pontryagin dual $\Z(p^\infty)$ of $N\cong \Z_p$ is divisible, it splits in  the Pontryagin dual of $N_2$. Therefore, $N_2=M\times N_1$, where $q\restriction_{N_1}:N_1 \to N$ is an isomorphism. Since $N_1 \cap G\ne 0$, it follows that $N \cap q(G)\ne 0$ as $q\restriction_{N_1}$ is an isomorphism.
\end{proof} 

%\begin{fact} Let $K$ be a topological Abelian group and let $H \leq G$ be subgroups of $K$. If $G$ is (locally, weakly) essential in $K$ and $H$ is (locally, weakly) essential in $G$, then $H$ is (locally, weakly) essential in $K$.
%\end{fact}

 Now we show that also local essentiality is preserved under inverse images: 

\begin{lemma}\label{Claim1} If $f: K \to L$ is a continuous surjective homomorphism of compact Abelian groups and $G$ is a (dense) locally essential subgroup of $L$, then $G_1:= f^{-1}(G)$ is a  (dense) locally essential subgroup of $K$.
\end{lemma}

\begin{proof} Let $V$ be a \nbd \ of 0 in $L$ witnessing local essentiality of $G$. Let us see that $U = f^{-1}(V)$ witnesses local essentiality of $G_1$ in $K$. Suppose that $N$ is a closed subgroup of $K$ contained in $U$. Then $f(N) \subseteq V$ is a closed subgroup $L$. If $f(N) = \{0\}$, then $N \leq \ker f \subseteq G_1$, so $N \cap G_1 \ne \{0\}$. If $f(N) \ne \{0\}$, there exists $0\ne f(x) \in f(N) \cap G$ with $x\in N$. Obviously, $x \in G_1 $ as well. Hence, again $N \cap G _1\ne \{0\}$.

 Concerning density, if $G$ is dense in $L$, then $G_1$ is dense in $K$ since $G_1 \geq \ker f$. 
\end{proof}

The following {\em (Total) Minimality Criterion} (item (i) due to Prodanov \cite{P1} and Stephenson, Jr. \cite{St} in the case of a compact group $G$), plays a vital role in the study of minimal groups (proofs can be found in \cite{P1,St,DPS}):

\begin{theorem}\label{Crit} {\rm \cite{B}} Let $G$ be a topological group and $H$ a dense subgroup of $G$. Then 
\begin{itemize}
  \item[(i)] {\rm \cite{ACDD2,B}}  $H$ is {(locally)}  minimal if and only if $G$ is minimal and $H$ is (locally) essential in $G$;
  \item[(ii)] {\rm \cite{DP2}}  $H$ is totally minimal if and only if $G$ is totally minimal and $H$ is totally dense in $G$.
\end{itemize}
\end{theorem}

If $G$ is a (locally) minimal Abelian group, then a in (i) density is not necessary in order to conclude that a subgroup $H$ of $G$ is (locally) minimal whenever it is (locally) essential in $G$. (Argue with the (locally) minimal closed subgroup $\overline{H}$ of $G$ in place of $G$ using the fact that (local) essentiality of $H$ in $G$ implies 
(local) essentiality of $H$ in the (locally) minimal group $\overline{H}$.) 

 By Lemma \ref{Neeew}, we have

\begin{theorem}\label{Crit1}{\rm \cite{DPS}} Let $G$ be a compact Abelian group and $H$ a dense subgroup of $G$. Then 
\begin{itemize} 
    \item[(i)] $H$ is (locally) minimal if and only if $H$ is weakly essential in $G$ and  
    (there exists a \nbd \ $V$ of $e$ such that)  $H$ contains all subgroups of $G$ isomorphic to $\Z(p)$ (and contained in $V$); 
    % and $G\cap N\neq \{0\}$  for all subgroup $N\cong \Z_p$ of $G$;
    \item[(ii)] $H$ is totally minimal if and only if $H$ contains all finite subgroup of $G$ and $G\cap N\nsubseteq pN$ for all subgroup $N\cong \Z_p$ of $G$,
\end{itemize}
where $p$ runs over all prime numbers.
\end{theorem}

\subsection{$\delta$-subgroups and co-NSS subgroup of the compact Abelian groups}\label{Sec:delta}\hfill

 We need to recall some properties of dimension for (locally) compact group.
% which will be used in the proof of the next proposition.  
All the fundamental dimension functions, dim, ind and Ind, coincide for these groups and for a closed subgroup $N$ of a  locally compact  group $G$ the equality 
$$
\dim G = \dim G/N + \dim N \eqno (\text{Y})
$$  
was proved by Yamanoshita \cite{Y}. In particular, $\dim G = \dim c(G)$ for a  locally compact  group $G$. Moreover, if $L\leq G$ is a closed subgroup, then $\dim G/L = \dim G/c(L) = \dim c(G) / c(L)$. Finite-dimensional connected locally compact groups are metrizable (in the Abelian case this follows from Pontryagin duality and the fact that finite-rank torsion-free Abelian groups are countable, in the non-Abelian case this follows from the fact that such a group $G$ is homeomorphic to $K \times \R^n$ for some compact connected subgroup $K$ of $G$ and some $n\in\N$, and the structure theory of connected compact groups). Following \cite{HM}, we write $\dim K:= w(K)$ when $\dim K = \infty$; in particular, $\dim K\leq \omega$ for all metrziable compact groups $K$. So, for any compact Abelian group $G$, we have  $\dim G=\dim c(G)=r(\widehat{G})$ \cite{DPS,HM}.

\medskip
 A topological group $G$ is said to be an {\em NSS group} (or briefly, {\em NSS}) if there exists a \nbd \ $U$ of the neutral element of $G$ containing only the trivial subgroup of $G$ (so, $G$ is NSS iff $\{e\}$ is locally essential in $G$).

\begin{definition}\label{Def:D}  Let $K$ be a compact group. A closed normal  subgroup $N$ of $K$ is called: 

(a) a {\em co-NSS  subgroup}, if $K/N$ is an NSS group.  

(b) {\rm \cite{DLLM} } a {\em $\delta$-subgroup} of $K$, if $K$ is Abelian, $N$ is totally disconnected and $K/N$ is a torus, i.e., isomorphic to some (possibly infinite) power  of $\T$.
\end{definition}

\begin{fact}\label{Fact0} 
 Let $K$ be a compact group. It is known that:

 (a) every \nbd \  of $e\in K$ contains a co-NSS subgroup (see \cite[Corollary 2.43]{HM} or  \cite[Proposition 11.5.4]{ADGB}). 
 
(b) if $K$ is Abelian, then it contains $\delta$-subgroups (see for example \cite{DPS,HM}), if additionally $\dim K< \infty$, all they are co-NSS subgroups; 
\end{fact}

\begin{remark}\label{Laast:Rem} (a) For a compact not necessarily Abelian group $K$ a closed normal subgroup $N$ is co-NSS
if and only if there exists a \nbd \ $W$ of $e_K$ such that whenever $H$ is a subgroup of $K$ contained in $W$, then $H \leq N$.
To check the necessity, use the fact that  if $N$ is co-NSS, then $K/N$ is NSS, so $K/N$ has a neighbourhood $O$ of $e_{K/N}$ witnessing the NSS property.
Let  $W=V\cap  q^{-1}(O)$, where $q: K \to K/N$ is the canonical homomorphism. Then every closed subgroup of $K$ contained in $W$ is contained already in $N$. 

% More often, we need only the following property weaker than co-NSS: a topological group $K$ is {\em NSS with respect to} the normal subgroup $N$ if there exists a \nbd \ $W$ of $e$ in $K$ such that whenever $H$ is a subgroup of $K$ contained in $W$, then $H \leq N$. then $K$ is NSS with respect to the closed normal subgroup $N$ if and only if $N$ is co-NSS (see Remark \ref{Laast:Rem}).  If $N=\{e\}$ is trivial, these weaker concepts coincide with the well known property NSS. 

% If $N$ is co-NSS, then $K/N$ will be metrizable and finite-dimensional. Therefore, $w(N) = w(K)$ whenever either $K$ is not metrizable or $N$ is infinite. Moreover, $K$ is {NSS with respect to} the normal subgroup $N$.  
%%Indeed, as $K/N$ is a Lie group, we can choose a \nbd \ $O$ of $e$ in $K/N$ containing no non-trivial subgroup. 
%Take a neighbourhood $O$ of the neutral element witnessing the NSS property and let  $W=V\cap  q^{-1}(O)$, where $q: K \to K/N$ is the canonical homomorphism. 
%Then $W$ witnesses the fact that every closed subgroup of $K$ contained in $W$ is contained already in $N$. 

To check the sufficiency, assume that there exists a \nbd \ $W$ of $e$ in $K$ such that whenever $H$ is a subgroup of $K$ contained in $W$, then $H \leq N$.
%if $K$ is NSS with respect to a normal subgroup $N$ and this is witnessed by a closed \nbd \ $W$ of $e\in K$, then $N$ contains all closed subgroups of $K$ that are contained in $W$. 
By Fact \ref{Fact0}(a), there exists a co-NSS subgroup $M\leq K$ contained in $W$, $M\leq N$. Therefore, $N$ itself is co-NSS in $K$.

(b) From the above argument, we may always assume $N\subseteq W$ for the pair $N, W$ as in (a). 

(c) The pair $N, W$ as in (b) witnesses simultaneously  the local essentiality of (dense) subgroups $G$ of $K$ in the 
following sense: a subgroup $G$ is locally essential in $K$ w.r.t. $W$ iff $G_1= G\cap N$ is  locally essential in $N$  w.r.t. $N=W\cap N$
(i.e., $G_1$ is essential in $N$). In such a case $G_1$ is a closed minimal subgroup of $G$ (being an essential subgroup of $N$), but need not be dense in $H$. 
%This allows one to invert item (c) of the above proposition as follows: if $G$ is a dense subgroup of a compact group $K$, then $G$ is locally minimal iff there exists a co-NSS closed subgroup $N$ of $K$ such that $G_1= N \cap G$ is an essential (hence, minimal) subgroup of  $G$.
\end{remark}

The following fact collects properties of the $\delta$-subgroups that can be easily deduced from Pontyagin-van Kampen duality theorem.

\begin{fact}\label{Fact:Delta1} For every  compact Abelian group $K$ the following conditions are equivalent: 
\begin{itemize}
\item[(a)] $\dim K < \infty$;

\item[(b)] there exists a $\delta$-subgroup $N$ of $K$ such that $K/N\cong \T^n$, where $n:=\dim K$; 

\item[(c)]  some $\delta$-subgroup $N$ of $K$ is co-NSS;

\item[(d)]  every $\delta$-subgroup $N$ of $K$ is co-NSS;

\item[(e)] some $\delta$-subgroup $N$ of $K$ is weakly essential; 

\item[(f)] every $\delta$-subgroup $N$ of $K$ is weakly essential. 
\end{itemize}
\end{fact}

 The next claim, using a very subtle property of the co-NSS subgroups, will be needed in the sequel. 
 %\NB\footnote{anticipated from a wrong place in front of Claim \ref{NewClaim} (that will maybe die :-).}
% It uses a very subtle property of the  compact Abelian groups $K$ related to their co-NSS subgroups $N$ which is decisive  in the description of local essentiality of subgroups $G$ of $K$ (reducing the matter to local essentiality of the subgroup $G\cap N$ of $N$).

\begin{claim}\label{NewClaim*} Let $K$ be a compact Abelian group and  let $G$  be a subgroup of $K$. Then the following are equivalent:

\begin{itemize}
  \item[(a)] $G$  is  locally essential in $K$; 
  \item[(b)] $G \cap N$ is locally essential in $N$ for every co-NSS subgroup of $K$;
  \item[(c)] $G\cap N$ is locally essential in $N$ for some co-NSS subgroup of $K$.
\end{itemize}
 If $\dim K < \infty$, then $G$  is  locally essential in $K$ if and only if $G\cap N$ is locally essential in $N$ for some (every) $\delta$-subgroup of $K$.
\end{claim}

\begin{proof} (a) $\Rightarrow$ (b) is obvious as co-NSS subgroups are closed, while  (b) $\Rightarrow$ (c) is trivial. 

To prove (c) $\Rightarrow$ (a) assume that $G\cap N$ is locally essential in $N$ for some co-NSS subgroup of $K$.
By Remark \ref{Laast:Rem} there exists a \nbd \ $W$ of $0$ containing $N$ such that the non-trivial 
%$G\cap N$ is locally essential in $N$ with respect to $W\cap N$.  Then a 
closed subgroups of $K$ contained in $W$ are also contained in $N$; so non-trivially intersect $G$.
That is, $G$ is locally essential in $K$. 

The last assertion follows from Fact \ref{Fact:Delta1}. 
%To prove the implication (c) $\Rightarrow$ (a) it is enough to show that $N$ is locally essential in $K$, in view of Lemma~\ref{Rem:29}.  Indeed, since $K/N$ is an NSS group, it contains a neighbourhood of the identity in which the only subgroup is the trivial one. Thus, $K$ admits a neighbourhood $W$ of $0$ with $N \subseteq W$ such that $N$ is the largest subgroup of $K$ contained in $W$, i.e., every subgroup of $K$ contained in $W$ is already contained in $N$.  Therefore, $N$ is locally essential in $K$ with respect to this neighbourhood $W$.
\end{proof}

%The equivalence of (a)--(c) holds also for not necessarily Abelian compact groups $K$ (see Remark \ref{Laast:Rem}).

%This is a relevant particular case of a more general folklore result: {\em every compact Abelian group $K$ contains a  closed totally disconnected subgroup $N$ (named a {\em $\delta$-subgroup} in \cite{DLLM}) such that $K/N$ is a torus}, i.e., some (possibly infinite) power  of $\T$  (see for example \cite[...]{DPS}). 
We collect below some useful properties of the $\delta$-subgroups. 

\begin{fact}\label{Fact1}
  If $N$ is a $\delta$-subgroup of a  compact Abelian group $K$, then 
\begin{itemize} 
 \item[(a)]  $K/N \cong \T^\kappa$ and $\kappa$ does not depend on the choice of $N$; 
 \item[(b)] If $N_1$ is any closed totally disconnected subgroup of $K$, then $N+ N_1$ is still a $\delta$-subgroup of $K$; 
 \item[(c)] $N + c(K) = K$; 
 % as the restriction of the canonical homomorphism $q: K \to K/N$ to $c(K)$ is surjective (consequently, $q'(N) =K/c(K)$, where  $q':K \to K/c(K)$ is the canonical homomorphism). 
 \item[(d)] if $K$ is a (necessarily) closed subgroup of a compact Abelian group $K_1$, then $N$ is a $\delta$-subgroup of $K_1$ if and only if $K_1/K$ is a torus; 
 \item[(e)] if $N_1$ is a $\delta$-subgroup of a  compact Abelian group $K_1$, then $N\times N_1$ is a $\delta$-subgroup of  $K\times K_1$; in particular, if $K_1$ is a totally disconnected  compact Abelian group, then $N\times K_1$ is a $\delta$-subgroup of $K\times K_1$;
 \item[(f)] $N \cap c(K)$ is a $\delta$-subgroup of $c(K)$; 
 \item[(g)] if $K$ is a finite-dimensional connected compact Abelian group, then:  
\begin{itemize}
  \item[(g$_1$)] \cite[Lemma 4.2.8]{DPS} $N \cong \prod_p (\Z_p^{k_p} \times F_p)$, where $k_p \in \N$ and $F_p$ is a finite $p$-group with $r_p(F_p) + k_p \leq n = \dim K$.
  \item[(g$_2$)] any two $\delta$-subgroups $H_1$ and $H_2$ of $K$ are commensurable, i.e.,  $H_1 \cap H_2$ has finite index in both $H_1$ and $H_2$;
  \item[(g$_3$)] $N$ is weakly essential in $K$ and $K/M$ is NSS for any $\delta$-subgroups $M$;
  \item[(g$_4$)] if $N_1\leq N$ is a closed subgroup  of $N$, then $N_1$ is a $\delta$-subgroup of $K$ iff $N_1$ is open in $N$. 
\end{itemize}
\end{itemize} 
\end{fact}

\begin{proof} 
(a) Indeed, $\kappa =r(\widehat K)$ by a standard property of Pontryagin-van Kampen dulaity \cite{DPS}. 

(b) First of all, $N_2:= N+ N_1$ is still a totally disconnected compact subgroup of $G$. Moreover, $G/N_2$ is isomorphic to a quotient of $G/N$.
To realize that $N_2$ is a $\delta$-subgroup of $K$ it is enough to recall that quotients of tori are again tori. 

(c) The restriction $q_1$ of $q: K \to K/N$ to $c(K)$ is surjective by the general fact that surjective continuous homomorphisms of compact groups take the connected component of the domain onto the connected component of the codomain. The surjectivity of $q_1:c(K)\to K/N$ means $q(c(K)) = K/N$ which is equivalent to $N+c(K)= K$. 

(d) This follows from the topological isomorphism $K_1/K\cong (K_1/N)/(K/N)$ immediately.

(e) Obvious. For the second assertion note that if $K_1$ is a totally disconnected  compact Abelian group, then $K_1$ is its only $\delta$-subgroup.
(More precisely, if $N$ is the only $\delta$-subgroup of $K$, then $K$ is totally disconnected and $N=K$.)

(f) Let $q: K \to K/N \cong \T^\kappa$ be the canonical homomorphism. Then the restriction of $q$ to $c(K)$ sende $c(K)$ onto the connected compact group $K/N$, which must be a torus.

(g$_1$) Use the fact that $\widehat K$ is a discrete group, isomorphic to a subgroup $X$ of $\Q^n$ such that $\Z^n \leq X$, so that 
$X/\Z^n \cong \bigoplus_p  (\Z(p^\infty)^{k_p} \times F_p)$, where $k_p \in \N$ and  $F_p$ is a finite $p$-group with $r_p(F_p) + k_p \leq n = \dim K$.

(g$_2$) Since $K/H_1$ is a finite-dimensional torus and $H_2H_1/H_1$ is a totally disconnected closed subgroup of $K/H_1$, $H_2/(H_1\cap H_2)\cong H_2H_1/H_1$ must be finite. By symmetry, $H_1/(H_1\cap H_2)$ is also finite.

(g$_3$) This is straightforward.

(g$_4$) The natural map $K/N_1 \to K/N$ is a quotient homomorphism between connected compact Abelian groups.  Since $\dim K < \infty$ and $K/N$ is a torus, it follows that $K/N_1$ is a torus if and only if the kernel $N/N_1$ is finite.\\
Finally, note that the sufficiency holds true also without the restraint $\dim K < \infty$.
\end{proof} 

 It follows from item (b) that if $N$ is a $\delta$-subgroup of a  compact Abelian group $K$, then every  totally disconnected closed subgroup $N_1$ of $G$ containing $N$ is still a  $\delta$-subgroup. In the opposite direction, if $N$  is an $\delta$-subgroup of $K$ and $N_1$ is an open subgroups of $N$, then $N_1$ need not be a $\delta$-subgroup of $K$ unless $K$ is connected, as in (g$_4$)  above (take $K = \T \times \Z_p$, $N=\{0\} \times \Z_p$ is a $\delta$-subgroup of $K$, but its open subgroup $N_1 = \{0\}\times p\Z_p$ is not a $\delta$-subgroup of $K$).  

\begin{remark} 
If $G$ is a critical locally minimal precompact Abelian group with completion $K$, then for every 
$\delta$-subgroup $N$ of $K$ the closed subgroup $N_1:= N \cap G$ of $G$ still has the property 
$G/N_1 \cong \T^\kappa$, but this isomorphism is not topological any more. Indeed, the 
open canonical homomorphism $K \to K/N \cong \T^\kappa$ need not be open when restricted to $G$ (this 
occurs precisely when $\overline{N_1}$ coincides with $N$. So to algebraic isomorphism $G/N_1 \to \T^n$
is continuous, but need not be open. The completion of $G/N_1$ must be $K/  \overline{N_1}$. 
\end{remark}

Recall first the following well-known fact (see \cite{DPS,HM}). 

\begin{fact}\label{torisplit}
Every torus in a compact Abelian group topologically splits.
\end{fact}

%%%%%%%%%%%%%%%%%%%%%%%%%%%%%%%
%%% ------------ hidden proof ------------------
%
%For the reader's convenience, we provide a proof.
%
%\begin{proof} Let $K$ be a compact Abelian group and let $T$ be a torus in $K$.   Then $\widehat{T}$ is a free Abelian group, so the short exact sequence
%$$\{0\} \to \widehat{K/T} \to \widehat{K} \to \widehat{T} \to \{0\}$$
%splits.   Taking Pontryagin duals, the dual sequence
%$$\{0\} \to T \to K \to K/T \to \{0\} $$
%also splits.  Hence, the torus $T$ is a direct summand of $K$. \end{proof}}
%%%%%%%%%%%%%%%%%%%%%%%%%%%%%%%%%%

One can assign to a compact group $K$ three sets of primes as follows: 

\begin{definition} For a compact  Abelian group $K$ let 
$$
\pi(K)=\{p\in \P: (\exists N\leq K) N \cong \Z_p\}, \  \pi_{tor}(K)=\{p\in \mathbb{P}: K[p]\neq \{0\}\} \ \mbox{ and } \ \pi^*(K) = \pi(K) \cup \pi_{tor}(K).
$$
\end{definition}
Obviously, $\pi^*(K)=\{p\in \mathbb{P}: td_p(K)\neq \{0\}\}$. 

%\NB\footnote{I changed the formal definition of $\pi^*$  (avoiding the somewhat hostile/unknown [to may readers] notion $td_p(-)$) and chose the more natural order to define these three sets. I  hope that now this definitions are easier to remember.}

\begin{fact}\cite[Proposition 4.1.6]{DPS}\label{Conp*} 

(a) If $K$ is a non-trivial connected compact Abelian group, then $td_p(K)$ is dense in $K$, for every prime $p$.
% so $\pi^*(K) = \P$.

(b) $\pi^*(K) = \P$ for every compact Abelian group $N$ with $c(K) \ne \{0\}$.

(c) From (b), if $\pi^*(K)\neq \P$ for a compact Abelian group $K$, then $K$ is totally disconnected.
\end{fact}

It is a folklore fact that $\pi(K)=\emptyset$ when $K$ is a Lie group (since a closed subgroup of a Lie groups is a Lie group itself). The compact groups $K$ with $\pi(K)=\emptyset$ were introduced in \cite{DSt} under the name {\em exotic Lie groups} in order to (positively) answer the question of Prodanov on whether 
the exotic Lie groups are finite-dimensional and were further  studied in \cite{D-density} and  \cite{DS_Lie}. In the Abelian case the exotic Lie groups
were introduced and thoroughly studied much earlier in \cite{DP2} under the name {\em exotic tori} (where they were shown to be the completions of the torsion
minimal Abelian groups). 

  \begin{theorem}\label{Thm:ex:tori}\cite{DP2}  For a compact Abelian group $K$ the following are equivalent: 
  \begin{itemize}
   \item[(a)] $K$ is an exotic torus $($i.e., $\pi(K) = \emptyset)$; 
   \item[(b)] $\dim K < \infty $ and every $\delta$-subgroup $N$ of $K$ is isomorphic to $\prod_p B_p$, where $B_p$ is a bounded compact $p$-group for every $p \in \P$; 
   \item[(c)] $\dim K < \infty $ and there exists a $\delta$-subgroup $N$ of $K$ as in (b); 
  % $N \cong \prod_p B_p$ of $K$, where $B_p$ is a bounded compat $p$-group for every $p \in \P$; 
   \item[(d)] $t(K)$ is totally minimal; 
   \item[(e)] $t(K)$ is minimal. 
  \end{itemize}  
  \end{theorem}
  
    Clearly, one can add to these equivalent conditions also ``(f) the subgroup $\{0\}$ is weakly essential in $K$." From this observations and Lemma \ref{Rem:29} one obtains
  
  \begin{corollary}\label{Cor:13May} A  closed subgroup $N$  of a compact Abelian group $K$ is weakly essential if and only if $K/N$ is an exotic torus. In such a case  $\dim K/N < \infty$. \end{corollary}

\begin{definition}\label{hybridtori}
A compact Abelian group $K$ admitting a short exact sequence
$$
\{0\} \longrightarrow \prod_{k=1}^m \Z_{p_k} \longrightarrow K \longrightarrow \T^n \longrightarrow \{0\}
$$ 
is called a {\em hybrid torus}. When $n=0$, i.e., $K\cong \prod_{k=1}^m \Z_{p_k}$, we call $K$ a \emph{non-Archimedean torus}.
\end{definition}

%In view of Lemma \ref{torisplit}, 
%extensions of Archimedean tori by non-Archimedean ones split, so nothing new can be obtained. 
Hybrid tori are finite-dimensional, but need not be connected. Clearly, for a hybrid torus $K$ the closed subgroup $N$, that is a non-Archimedean torus such that $K/N$ is a torus, is a $\delta$-subgroup (but not all $\delta$-subgroups of a hybrid torus are necessarily non-Archimedean tori). 

Obviously, the groups that are simultaneously exotic tori and hybrid tori are precisely the tori. 

 One can describe {\em internally} a connected hybrid torus by means of the set-theoretic invariant $\pi^*(-)$. 
 
\begin{proposition}\label{prop:ht} For a compact connected Abelian group $K$ of finite dimension, the following are equivalent: 
  
  \begin{itemize}
  \item[(a)] $K$ is a hybrid torus;

  \item[(b)] $\pi^*(N)$ is finite for some $\delta$-subgroup $N$ of $K$; 
    
  \item[(c)] $\pi^*(N)$ is finite for every $\delta$-subgroup $N$ of $K$.    
  \end{itemize}
\end{proposition}

\begin{proof} 

(a) $\Rightarrow $ (b) If $N$ is a non-Archimedean torus in $K$  witnessing that $K$ is a hybrid torus, then $N$ is a $\delta$-subgroup of $K$ and obviously $\pi^*(N)$ is finite. 

(b) $\Rightarrow $ (c)  Since $K$ is finite-dimensional, for any two $\delta$-subgroups $H_1$ and $H_2$ of $K$ the subgroup $H_1 \cap H_2$ has finite index in both $H_1$ and $H_2$, by Fact \ref{Fact1}(g$_2$). Hence, if $\pi^*(N)$ is finite for some $\delta$-subgroup $N$ of $K$, then $\pi^*(H)$ is finite for every $\delta$-subgroup $H$ of $K$. 

(c) $\Rightarrow $ (a) Let $N$ be a $\delta$-subgroup of $K$. Since $\pi^*(N)$ is finite, it follows from Fact~\ref{Fact1}(g$_1$) that
$$
N \cong \left(\prod_{p \in \pi^*(N)} \Z_p^{k_p}	\right) \times F,
$$
where $0 \leq k_p \leq n := \dim K$ for each $p$, and $F$ is a finite group. Let $N_1 \leq N$ be the closed subgroup given by $N_1 \cong \prod_{p \in \pi^*(N)} \Z_p^{k_p}.$
Then $N_1$ is a non-Archimedean torus and $N=N_1 \times F$.
%\NB\footnote{was: By Fact \ref{Fact1} (g$_4$), the open subgroup $N_1$ of $N$ is a $\delta$-subgroup of $K$. }
Hence the connected group $K/N_1$ has a finite subgroup $N/N_1 \cong F$, such that $(K/N_1)/(N/N_1) \cong K/N \cong\T^n$. 
Therefore, $K/N_1$, being locally isomorphic to $\T^n$ is a connected compact Abelian Lie group of dimension $n$, hence $K/N_1 \cong \T^n$.
 Therefore, $K$ is a hybrid torus.
%%%%%%%%%%%%%%%%%%%%%%%%%
%Let $X = \widehat K$. Then $X$  is a discrete group, isomorphic to a subgroup $X$ of $\Q^n$ such that $\Z^n \leq X$, so that $ n = \dim K$, 
%$X/\Z^n \cong \bigoplus_p  (\Z(p^\infty)^{k_p} \times F_p)$, where $k_p \in \N$ and 
%$F_p$ is a finite $p$-group with $r_p(F_p) + k_p = n$. As $N = (\Z^n)^\bot$ is a $\delta$-subgroup of $K$, $\pi^(N)$ is finite. 
%Since $N \cong \widehat{X/\Z^n}$, this means that $t_p(X/\Z^n) \ne 0$ only for finitely many primes $p$. 
%Hence, for some finite set $P$ of primes, $X/\Z^n \cong F\oplus \sum_{p\in P}\Z(p^\infty)$, where $F$ is a finite subgroup of $X/\Z^n$. 
%Taking a larger free subgroup $S$ of $X$, with $\Z^n \leq S \leq X$, we can assume without loss of generality that $X/S \cong  \oplus \sum_{p\in P}\Z(p^\infty)$. 
%Now $N_1 = S^\bot \cong \prod_{p\in P}\Zp$ is a non-Archimedean torus in $K$ which is still a $\delta$-subgroup of $K$. Hence, $K$ is a hybrid torus. 
%%%%%%%%%%%%%%%%%%%%%%%%%
\end{proof} 

\begin{remark}\label{New:Rem}
Using the above description of connected hybrid tori, it is easy to see that the quotient of a connected hybrid torus
is a again a (connected) hybrid torus. This is no more true without the assumption of  connectedness (e.g., $\Z_p$ is a non-Archimedean
torus, hence a  hybrid torus as well, yet its quotient $\Z_p/p\Z_p\cong \Z(p)$ is not a hybrid torus). 
\end{remark}

\subsection{Hereditarily disconnected quotients of (locally) minimal groups}\label{Special:quot}
\hfill

We recall here three remarkable (functorial) subgroups attached to every  topological group $G$, measuring connectedness properties: 

\begin{itemize}
  \item  $c(G)$ denotes the connected component of the identity;
  \item  $q(G)$ denotes the {quasi-component} of the identity, that is, the intersection of all clopen sets containing the identity;
  \item  $o(G)$ denotes the  intersection of all open subgroups of $G$.
\end{itemize}

Clearly, $c(G) \subseteq  q(G) \subseteq o(G)$, and all three subgroups are normal and coincide for locally compact groups.  Moreover,  $o(G) = G \cap c(\widetilde G)$ for a precompact group $G$, as $c(\widetilde G) = q(\widetilde G) = o(\widetilde G)$. If $G$ is pseudocompact, then $q(G) = o(G) = G \cap c(\widetilde G)$ \cite{D-CR,Ddim}. 
Therefore, if $c(G) = c(\widetilde G)$ for a precompact group $G$, then these coincide also with $q(G)=o(G)$ as well. 

 The next more subtle fact follows from the main results of \cite{Ddim} (see there Corollary  1.3 and Theorem 1.7; see also \cite{D-CR}): 

\begin{fact}\label{Fact:June6} For a \cc, \mi\ Abelian group $G$ the quotient $G/c(G)$ is minimal and zero-dimensional and $c(G)$ is dense in $c(\widetilde G)$. 
\end{fact}

Following the terminilogy from \cite{Eng}, we call a topological group $G$ {\em hereditarily} (resp., {\em totally}) {\em disconnected}, if $c(G) = \{e\}$ (resp., $q(G) = \{e\}$). By what we said above, these two properties coincide for locally compact groups.

The following fact was given in \cite{DPS} as Exercise 4.5.15 without neither a proof nor a hint: 

\begin{fact}\label{qmini}
If $G$ is a minimal Abelian group and $C = c(\widetilde G)$, then $G/(G\cap C)$ is minimal iff $\overline{G\cap C}=C$. 
\end{fact}

Since this fact was used without proof in several papers afterwards (e.g., \cite[Theorem 1.7]{Ddim}, \cite[Theorem 3.6(a)]{DTk}, only in
\cite{DU} a proof is given under stronger assumptions), we propose here a very brief sketch of a proof. 

 Indeed, we shall prove the stronger result:
 
\begin{fact}\label{New:Fact} If $G$ is a minimal Abelian group and $C$ is a connected compact subgroup of $\widetilde{G}$, then $G/(G\cap C)$ is minimal iff $\overline{G\cap C}=C$. 
\end{fact}

\begin{proof} First we note that the hypothesis $\overline{G\cap C}=C$ is equivalent to asking the algebraic and continuous embedding of $G/(G\cap C)$ into $\widetilde G/C$ to be also topological. This proves, among others, the implication $\Rightarrow$.

After this first step, to ensure the implication $\Leftarrow$
one has to note that divisibility of $C$ and the minimality of $G$ ensure that $G/(G\cap C)$ contains the socle of $\widetilde G/C$
(see Claim \ref{LAST:claim} below) and $G/(G\cap C)$ is weakly essential in $\widetilde G/C$ (Lemma \ref{Rem:29}). \end{proof}

\begin{claim}\label{LAST:claim}
Let $p\in \P$, $K$ be a compact Abelian group and $H$ be a closed subgroup. Then for the canonical homomorphism $q: K \to K/H$ one has $q(K[p]) = (K/H)[p]$ in each of the following cases: 
\begin{itemize}
    \item[(a)] $H$ is $p$-divisible; 
    \item[(b)] $H$ splits algebraically; 
    \item[(c)] $p\notin \pi^*(H)$ (i.e., $td_p(H)=\{0\}$).
\end{itemize}
When $H$ is connected, (a) and (b) are automatically satisfied for any $p$, so $q(\Soc(K))=\Soc(K/H)$.
\end{claim}

\begin{proof}Obviously, $q(K[p])\leq (K/H)[p]$. For the reverse inclusion it suffices to check that $(K/H)[p] \leq q(K[p]) $.

(a)  Pick $x\in (K/H)[p]$. Hence, $px=0$, so if $x= y + H$ with $y\in K$, one has $py \in H$. Since $H$ is $p$-divisible, $py = pz$ for some $z \in H$. Then $t= y-z \in K[p]$, as $pt=0$. Therefore, $x=q(y) = q(t+z)=q(t) \in q(K[p])$.  

(b) In pure algebraic sense, $K\cong (K/H)\oplus H$, with $q$ the projection onto its first factor.
So $$q(K[p])=q((K/H)[p]\oplus H[p])=(K/H)[p].$$

(c) It suffices to see that if $H$ satisfies (c), then $H$ is $p$-divisible, so (a) applies. Indeed, if $p\notin \pi^*(H)\supseteq \pi^*(c(H))$, then 
$\pi^*(c(H)) \ne \P$, so $H$ is totally disconnected (by Fact \ref{Conp*}), hence $H = \prod_{p'\in \P} td_{p'}(H)$. 
%$H = \prod_{q\in \P} td_q(H)$ is totally disconnected, by Fact \ref{Conp*}. 
Since $td_p(H) = \{0\}$, only the components $ td_{p'}(H)$ with $p'\ne p$ appear in the product and all they are $p$-divisible. 
\end{proof}

Without any other additional condition on $G$, one cannot deduce anything on the quotient $G/c(G)$, since 
$c(G)$ in general is (much) smaller than $G\cap C = o(G)$ (see above).  Of course, imposing $G \supseteq c(\widetilde G)$
one immediately has $G\cap C = c(G)$, so that the above fact can be applied. 

The following obvious corollary of this fact appeared in the literature (e.g., \cite[Theorem 3.6(a)]{DTk}): 

\begin{corollary}\label{Coro:MinG/c(G)}
If $G$ is a minimal Abelian group and $G \supseteq c(\widetilde G)$, then 
$c(G) = c(\widetilde{G})$ is compact and $G/c(G)$ is minimal. 
\end{corollary}

 Here we give another, somewhat unexpected, corollary of the claim: 

\begin{corollary}\label{NEW:corollary} 
If $K$ is a compact Abelian group, then the subgroup $c(K)$ of $K$ is essential if and only if $c(K) =K$, i.e., $K$ is connected. 
\end{corollary}

\begin{proof} Assume that $c(K)$ is essential in $K$. Then $\Soc(K) \leq c(K)$. To the canonical homomorphism $q: K \to K/c(K)$ apply the last assertion of Claim \ref{LAST:claim}. It gives $\Soc(K/c(K)) = q(\Soc(K))=\{0\}$. Hence, $K/c(K)$ is torsion-free. On the other hand, since $c(K)$ is (weakly) essential in $K$ by Lemma \ref{Rem:29}, we conclude  that  $\{0\}$ is weakly essential in $K/c(K)$, which means that $K/c(K)$ has no infinite $p$-monothetic subgroups.
In conjunction with our previous conclusion that $K/c(K)$ is torsion-free, this means that $K/c(K)$ is trivial. Therefore, $c(K) = K$.  
\end{proof}

\section{Approximating local minimality by minimality}\label{Subsec:approx}\label{LARGE:SUBGROUPS}

In the initial part of this section we show that, roughly speaking, a locally minimal Abelian group has a ``big" minimal (hence, precompact) subgroup. 

%\subsection{Large minimal subgroups of the local minimal Abelian groups}

Realizing that for a compact torsion-free Abelian group $G$ weak  essentialily and essentialily of its (dense) subgroups are equivalent
% the first case in item (i) of Theorem \ref{Crit1} does not occur and  for any $N\cong \Z_p$ of $G$ the intersection $N \cap V$ is an open subgroup of $N$, so contains a subgroup $N_1 = p^n N$ for some $n\in \N$, 
 we obtain the following immediate corollary from Theorem \ref{Crit1}: 

\begin{corollary}\label{New:Corollary} If $G$ is a compact torsion-free Abelian group and $H$ a dense locally minimal subgroup of $G$, then $H$ is minimal. 
\end{corollary}

%Here ``torsion-free" imposed on the completion on $G$ cannot be relaxed to imposing this restraint on $G$: every torsion-free discrete group $G$ is also locally minimal, but fails to be minimal when $G$ is infinite. 

\smallskip 

Call a subgroup $A$ of an Abelian group $X$ \emph{superfluous} if, for every subgroup $B$ of $X$, the equality $A + B = X$ implies that $B = X$. We denote this shortly by $A\leq _s X$.

\begin{example}\label{Exa:july1}  For every $n\in \N$, $\Z^n \leq_s \Q^n$, since $\Z^n + B=\Q^n$ for $B \leq \Q^n$ implies that $\Q^n/B$ is both divisible and 
finitely generated, so $\Q^n/B = \{0\}$ and $B = \Q^n$. \end{example}

\begin{lemma}\label{Le:finiterank} Let $X$ be an infinite Abelian group. 
\begin{itemize}
    \item[(a)] If $S\leq H$ is a chain of subgroups of  $X$ with either $S \leq_s H$ or $H \leq_s X$, then also $S\leq_s X$. 
    \item[(b)] If $X$ is torsion-free with $r(X)$ infinite and $A\leq_s X$, then $r(A)$ is finite.
    \item[(c)] $A\leq_s X$ if and only if the annihilator $A^\bot$ of $A$ is essential in $\widehat{X}$.
    \item[(d)] If $Z\leq _s X$, then $|X/Z|=|X|$.  
\end{itemize}
\end{lemma}

\begin{proof} (a) Assume that $S\leq_s H$ and $S + B=X$ for some $B\leq X$. Then obviously $S + (B\cap H) = H$, so $B\cap H=H$, thus $S \leq H \leq B$. Therefore, $B = X$. The case when $H\leq_s X$ is trivial.

(b) The divisible hull $D:=D(X)$ of $X$ is torsion-free of rank $\sigma = r(X)$  infinite. Hence, $D$ can be identified with $\bigoplus_\sigma \mathbb{Q}$. By (a), $A\leq_s D$. Assume for a contradiction that $A$ has a countably infinite independent subset $C$. Fix an element $y_0\in C$ and let $Y=\mathbb{Q}y_0$. We also fix a bijection $f: C\to Y$ such that $f(y_0)=y_0$. Since $C$ is independent (over the filed $\mathbb{Q}$), there is a linear mapping $\widetilde{f}: D\to Y$ extending $f$. Moreover, the restriction of $\widetilde{f}$ to $Y$ is the identity automorphism because $\widetilde{f}(y_0)=f(y_0)=y_0$. So, $D=Y\oplus Z$, where $Z=\ker \widetilde{f}$. Now any $x\in D$ can be represented as $y+z$ with $y\in Y$ and $z\in Z$. Take $c\in C$ such that $f(c)=y$, then $x-c\in Z$ and hence $x\in c+Z\subseteq C+Z$. This is saying that $D=C+Z=A+Z$. But $Z$ is a proper subgroup of $D$, this contradicts $A\leq _s D$.

(c) Assume that $A^\bot$ is essential in $\widehat X$ and $Z$ is a proper subgroup of $X$. Then  $Z^\bot \ne \{0\}$ is a closed subgroup of $\widehat X$, so
 $A^\bot \cap Z^\bot\ne \{0\}$. Therefore, $(A + Z)^\bot = A^\bot \cap Z^\bot \ne \{0\}$, implying that  $A+Z \ne X$. 
 
Now assume that $A\leq_s X$. To prove that the closed subgroup $A^\bot$ of $\widehat X$ is essential pick a non-trivial closed subgroup $N$ of $\widehat X$. Then $N^\bot$ is a proper subgroup of $X$, hence $A + N^\bot \ne X$. Therefore, $\{0\}\ne (A + N^\bot)^\bot = A^\bot \cap N$ witnessing essentiality of $A^\bot$.

(d) Let $Y = X/Z$ and assume for a contradiction that $|Y| < |X|$  and let $q: X \to Y = X/Z$ be the surjective canonical homomorphism with  kernel $Z$.

Let us see first that $Y$ cannot be finite. Indeed, assume that  $Y$ is finite and take a finite subset $F$ of $X$ such that $q$ bijectively maps $F$ onto $Y$, so $X=F+Z$. Then $X = B + Z$ as well, where $B=\langle F \rangle$ is the subgroup generated by $F$. 
Hence, $B=X$, since $Z$ is superfluous in $X$. Then $X$ is finitely generated, so $X \cong \Z^m \times H$, for some $m\in \N_+$ and some finite Abelian group $H$.

Take a prime number $p$ with $p>|Y|$. Let us show that $F\subseteq pB+Z$. For any $x\in F$, there exists a natural number $n$ dividing $|Y|$ such that $nx\in Z$.
Since $p> |Y|$, $p$ must be coprime with $n$, so there exist integers $u$ and $v$ such that $1=up+vn$. As $px \in pB$ we deduce that   
$$
x=u(px)+v(nx)\in pF+Z\subseteq pB+Z,
$$
implying that $F\subseteq pB+Z$, and consequently $X = B  \subseteq pB+Z$. Since $Z$ is superfluous, we deduce that $pB =X$. %As $X=B$, 
This proves that $X=B$ is $p$-divisible, a contradiction in view of $X \cong \Z^m \times H$. 

Let us see now that our assumption $|Y| < |X|$ leads to contradiction also when $Y$ is infinite. There exists a subgroup $A$ of $X$ such that $|A| = |Y|$ and $q(A) = Y$, so $A + Z=X$. On the other hand, as $Z\leq_s X$, $|A| = |Y| < |X|$ implies that $A$ is proper, so $A+Z \ne X$, a contradiction.
\end{proof}
  
\begin{proposition}\label{NEW:proposition} If $M$ is a closed subgroup of a compact Abelian group $K$, then $w(c(K)) = w(c(M))$ provided one of the following two conditions are satisfied: 
 \begin{itemize}
  \item[(a)] $M$ is a $G_\delta$-subgroup of $K$ with $c(M)\ne \{0\}$; 
  \item[(b)] $M$ is a locally essential subgroup of $K$ and $\dim K$ is infinite.
 \end{itemize}
\end{proposition}

  \begin{proof}
  % \NB\footnote{This preamble with the reduction was missing. } 
  Let $C = c(K)$. In both (a) and (b) we can assume that $K = C$ is connected, by replacing $M$ with $M_c:= M\cap C$.  Indeed, $c(M_c) = c(M)$ as $c(M) \leq M_c$. Moreover,   in (a) $M_c$ is a $G_\delta$-subgroup of $C$, while in (b) $M_c$ is a locally essential subgroup of $C$ and $\dim C = \dim K$ is infinite.

(a)  If $K$ is metrizable, the conclusion $w(c(M)) = w(K)$ is immediate since $c(M) \ne \{0\}$. Hence, we assume $K$ is not metrizable.

Let $X=\widehat{K}$, then it is torsion-free and $|X|=w(K)$ is uncountable.
Since $M$ is a $G_\delta$-subgroup of $K$, $K/M$ is metrizable. Thus, $Y:=M^\perp \cong \widehat{K/M}$ is countable.
Put $Z=X/Y$. We have $\widehat{M} \cong Z$ and $c(M) \cong \widehat{Z/t(Z)}$. It suffices to show $|X| = |Z/t(Z)|$.
Let $q: X\to Z$ be the canonical map. The saturation $\mathrm{sat}(Y) = q^{-1}(t(Z))$ is countable since $X$ is torsion-free and $Y$ is countable.
Since $X / \mathrm{sat}(Y) \cong Z / t(Z)$ and $X$ is uncountable while $\mathrm{sat}(Y)$ is countable, we conclude that $|X|=|Z/t(Z)|$.
   
  (b) We first consider the case that $M$ is essential in $K$.   We are going to prove a stronger assertion, namely $\dim K/c(M) < \infty$, which implies $w(c(M))=\dim c(M) = \dim K =w(K)$, in view of (Y) and our assumption that $\dim K$ is infinite.
  
 Let $A=M^\perp$. By Lemma  \ref{Le:finiterank}(c), $A\leq_s  \widehat K$ since $M$ is essential in $K$. As $ \widehat K$ is torsion-free, $r(A)$ is finite by item (b) of the same lemma. This is equivalent to say that $\dim K/M$ is finite. The desired inequality $\dim K/c(M) < \infty$ follows from (Y) and the topological isomorphism
$K/M\cong (K/c(M))/(M/c(M)),$ where $M/c(M)$ is zero-dimensional.

For the general locally essential case we take a \nbd\ $U$ of $0$ in $K$ witnessing local essentiality of $M$ and a co-NSS subgroup $N$ of $K$, so $\dim  K/N < \infty$. 
%such that  $N \subseteq U$ and $K/N$ is a Lie group. 
Then, $w(c(N))=\dim N=\dim K$ is infinite by (Y) and $M\cap N$ is essential in $N$. It follows from the above argument  applied to $M\cap N \leq N$ 
that $w(c(M\cap N))=w(c(N))=w(K)$. Hence $w(c(M)) \geq w(c(M\cap N)) = w(K)$. Since the inequality $w(c(M))\leq w(K)$ is obvious, we are done.  
% Now, using these facts, we can assume the group $C$ in our hypothesis to be connected with \NB infinite $\sigma = \dim C = w(C)$. 
% We are going to prove that $\dim C/c(L) < \infty$, which implies $\dim c(L) = \dim C  $, in view of (Y).   
% By a well known fact, there exists a continuous surjective homomorphism $f: \widehat{\Q}^\sigma \to C$ such that $\ker f$ totally disconnected. 
% Let us see note now that $L_1$ is an essential closed subgroup of $\widehat{\Q}^\sigma$. Moreover, the restriction $f': L_1 \to L$ is continuos and surjective. 
%So $f'(c(L_1)) = c(L)$ and $\dim c(L) = \dim c(L_1)$. Therefore, the proof of the claim for the pair $L_1 \leq  \widehat{\Q}^\sigma$ will be sufficient, since
%$C/L\cong \widehat{\Q}^\sigma/L_1$. 
%Let $A=L_1^\perp\leq \bigoplus_\sigma \Q$.By a standard Pontryagin duality argument, $A$ is a subgroup of $\bigoplus_\sigma \Q$ satisfying the property in Lemma \ref{Le:finiterank} since $L_1$ is essential in $\widehat{\Q}^\sigma$; so $r(A)$ is finite.This is equivalent to say that $\dim  \widehat{\Q}^\sigma/L_1$ is finite. 
 \end{proof}

\begin{example} For $n\in \N_+$ the group $K = \widehat \Q^n \cong \widehat{\Q^n}$ is connected, and $M:= (\Z^n)^\bot$ is an essential subgroup of $K$
in view of Example \ref{Exa:july1} and Lemma \ref{Le:finiterank}(c). Nevertheless, $w(K)> w(c(M))$, since 
$c(M) = \{0\}$. Hence, the second condition in item (b) of the above theorem cannot be relaxed.  
\end{example}
 
\begin{proposition}\label{Prop:lin} Let $G$ be a locally minimal Abelian group. Then: 
\begin{itemize}
  \item[(a)] $G$ contains a closed minimal $G_\delta$-subgroup $G_1$; 
  \item[(b)]  if $G$ has a linear topology, then $G_1$ can be chosen to be open; 
  \item[(c)]  if $G$ is precompact and non-metrizable, then $G_1$ can be chosen with $w(G_1)=w(G)$. 
%  \item[(d)]  if $G$ is precompact and either $\dim \widetilde G$ is infinite or $c(\widetilde G_1)\ne \{0\}$, then $w(c(\widetilde G_1)) = w(c(\widetilde G))$. 
\end{itemize}
In particular, a linearly topologized Abelian group is locally minimal if and only if it contains on open minimal subgroup. 
\end{proposition}

\begin{proof} Denote by $K$ the completion of $G$, which is locally minimal, by Theorem \ref{Crit}(i). In item (c) $K$ will be compact. 

(a) By \cite[Corollary 3.3]{DDHXX}, there exists an open \nbd \ $U$ of 0 in $G$ such that every closed subgroup $H$ of $G$ contained in $U$ is minimal.  In order to get the desired $G_1$ let $U_0:= U\cap -U$ and for each $n>0$ choose an open symmetric \nbd \ $U_n$ of 0 in $G$ such that $\overline{U_n+U_n} \subseteq U_{n-1}$. Then $G_1 = \bigcap U_n$ is a closed  $G_\delta$-subgroup of $G$ contained in $U$, so $G_1$ is minimal.  

(b) This follows immediately by the proof of (a), since $U$ can be chosen to be an open subgroup and $U_n$ can be chosen to be $U$, so that $G_1=U$ will be open. 

(c) Assume that $G$ is not metrizable, hence $K$ is not metrizable as well. By Theorem \ref{Crit}(i), there exists an open \nbd \ $V$ of 0 in $K$ witnessing local minimality of $G$. By Remark \ref{Laast:Rem}, there exists a co-NSS subgroup $N$ of $K$ contained in $V$. Then $N$ is a $G_\delta$-subgroup of $K$ and $w(N)=w(K)$.
Moreover, $G_1:=G\cap N$ is essential in $N$, by Remark \ref{Laast:Rem}.
Then, $\overline{G_1}$, the closure of $G_1$ in $K$, is also essential in $N$.
By Lemma \ref{Le:finiterank}(c) and (d), we obtain that $w(\overline{G_1})=w(N)$.
Then we have 
$$w(G_1)=w(\overline{G_1})=w(N)=w(K)=w(G).$$
It remains to note that $G_1$ is minimal since it is essential in $\overline{G_1}$.

 The final assertion follows from (b) and the fact that a group with an open minimal subgroup is locally minimal \cite[Proposition 2.4]{ACDD1}. 
 %In this case $K$ has a linear topology as well, we can choose the \nbd \ $V$ of 0 witnessing local essentiality of $G$ to be an open subgroup of $K$. Then $G_1 = G \cap V$ is an open subgroup of $G$ which is also a dense essential subgroup of $V$, hence $G_1$ is minimal. 
 \end{proof}

\begin{remark}\label{Rem:3.6} 
(i) The non-metrizability restraint in item (c) is necessary. There are infinite locally minimal subgroups $G$ of $\T$ (as $\Z(p^\infty)$) such that the only minimal subgroups $G_1$ of $G$ are finite.   
   %So $W$ and $N$ contain the same closed subgroups, therefore one can replace, in the test of local minimality, the neighbourhood $W$ by the subgroup $N$. Therefore, $G$ is locally essential w.r.t. $W$ iff $\Soc(N) \leq G$ and $G\cap N$ is weakly essential in $N$.

(ii) It is known that a group containing an open minimal subgroup is  locally minimal \cite[Proposition 2.4]{ACDD1}.  Item (b) of the above proposition shows that this sufficient condition becomes also necessary when the group in question  has linear topology. 
 %(e.g., when it is precompact and its completion is totally disconnected). 
\end{remark} 

 Here are two immediate applications of  Proposition \ref{Prop:lin}, further applications will come in the sequel. 
 For an abstract Abelian group $G$ we denote by $G^\#$ the group $G$ equipped with the Bohr topology, namely the topology generated 
by all characters $G \to \T$.

\begin{corollary}\label{Coro1:May27} For an Abelian group $G$, the following are equivalent: 
\begin{itemize}
\item[(a)] $G^\#$ is locally minimal; 
\item[(b)] $G^\#$ is compact;
\item[(c)]  $G$ is finite.
\end{itemize}
\end{corollary}

\begin{proof} The implications (c) $\Rightarrow$ (b) $\Rightarrow$ (a) are obvious. To prove that (a) $\Rightarrow$ (c) assume that $G$ is infinite. Then $G^\#$ is a precompact group with $w(G^\#) = 2^{|G|}> \omega$ (the equality is well known, see \cite{D_bohr}). 
 Then by Proposition \ref{Prop:lin}(c), $G$ contains a closed minimal subgroup $G_1$ with $w(G_1)=w(G)$. Since minimal groups satisfy  $w(G_1) \leq |G_1|$, 
 we deduce that $w(G) = w(G_1) \leq |G_1| \leq |G| $, which contradicts the fact that  $w(G^\#) = 2^{|G|}$. 
\end{proof}

Locally compact groups are both complete and locally minimal. Now we see that for 
%%%%%%%%%%%% for  the record  S(X). %%%%%%%%%%%%%%%%%%%%%%%%
%\NB\footnote{Not true in the non-Abelian case. A perfect example could be the symmetric group $S(X)$ for some very large $X$. I am confident this $S(X)$
%cannot have compact $G_\delta$-subgroups. What do you think ? 
%\BH I have already added the example you sent to me in the last message\\
%\NB Very good, thanks ! Maybe we can recycle $S(X)$ if necessary somewhere (or just forget about it :-). 
%}%%%%%%%%%%%%%%%%
Abelian groups this 
implications can be inverted in some sense -- the complete locally minimal  Abelian group $G$ contains a large compact subgroup:   

\begin{corollary}\label{Coro2:May27} Every complete locally minimal  Abelian group $G$ contains a compact $G_\delta$-subgroup. 
\end{corollary}

\begin{proof} Let $G$ be a complete locally minimal  Abelian group $G$. By Proposition \ref{Prop:lin}(a), $G$ contains a closed minimal $G_\delta$-subgroup $G_1$. 
Then $G_1$ will be also complete, so compact by Prodanov-Stoyanov's theorem. 
\end{proof}

This corollary fails in the non-Abelian case. Indeed, it is shown in \cite[Example 4.3(b)]{DDHXX1} that there exists a non-compact, locally compact group $G$ whose arbitrary powers are minimal (and also complete). However, since every $G_\delta$ subgroup of $G^{\omega_1} $ contains a copy of $G^{\omega_1} $, it follows that $G^{\omega_1} $ cannot have an (even locally) compact $G_\delta$-subgroup.

%\subsection{}\label{Sec:Conditions}
\bigskip

We conclude this section with a theorem giving some necessary and some sufficient conditions for local essentiality of subgroups
of a compact Abelian group $K$. To formulate one of them more conveniently in terms of socles, for subgroups $G$ and $H$ of $K$ we introduce the set 
$$\pi_{(H:G)} =\{p\in \P: H[p]\not \leq G\}.$$ Clearly, $\pi_{(H:G)}  \subseteq  \pi_{(H_1:G_1)}$, whenever $H \leq H_1$ and $G_1\leq G$. 
In particular, $ \emptyset = \pi_{(\{0\}:G)} \subseteq \pi_{(H:G)}  \subseteq \pi_{(H:\{0\})} = \pi_{tor}(H)$. 

\begin{theorem}\label{claim4} Let $K$ be a compact Abelian group and $G$ be a weakly essential subgroup of $K$.
\begin{itemize} 
  \item[(a)]If $G$ is locally essential in $K$, then% $G$ is weakly essential in $K$ and 
      \begin{itemize}
         \item[($*$)]  for every $p\in \P$ the subgroup $G[p]$ of $K[p]$ is open, 
      \end{itemize}
  and there exits $n\in \N$ such that  $r_p(K[p]/G[p]) \leq n$ for every prime $p$.
  \item[(b)] If $(*)$ holds and $ |\pi_{(K:G)} |   <  \infty$, then $G$ is locally essential in $K$. 
  \item[(c)] If $K_1:= \overline{\Soc(K)}$ is finite-dimensional, then the following conditions are equivalent:
\begin{itemize}
    \item[(c$_1$)] $G$ is locally essential in $K$; 
    \item[(c$_2$)] for every $p\in \P$ the subgroup $G[p]$ of $K[p]$ is open and $|\pi_{(M:G)}|< \infty$ for every  $\delta$-subgroup $M$ of $K$.
    \item[(c$_3$)]  $G \cap N$ is locally essential in $N$ for some (every)  %co-NSS 
    $\delta$-subgroup $N$ of $K_1$. 
\end{itemize}
    In particular, when $K$ is totally disconnected, then $G$ is locally essential in $K$ iff $|\pi_{(K:G)}|< \infty$ and $G[p]$ is open in  $K[p]$ for every $p\in \P$.  
 \end{itemize}
\end{theorem}

\begin{proof} (a) Assume that $G$ is locally essential in $K$, witnessed by an open \nbd \ $W$ of 0. Then for every prime $p$, $S_p:= W\cap K[p]$ is an open \nbd \ of 0 in $K[p]$. Since $K[p]$ has linear topology, $S_p$ contains an open subgroup $V_p$ of $K[p]$ which is contained in $G$, as $V_p \subseteq W$. 
In particular,  $G[p]=G\cap K[p]$ is open in $K[p]$. Hence, $K[p] = G[p] \oplus F_p$ for some finite subgroup $F_p$ of $K[p]$.
%Then $G \cap N$ is essential in $N$ for any closed subgroup $N \leq K$ such that $N \subseteq W$.

Choose a co-NSS subgroup $N\leq K$ contained in $W$. Since $G$ is locally essential in $K$ with respect to $W$,  $G \cap N$ is essential in $N$. %, in view of Remark \ref{Laast:Rem}. 
  So, from $F_p \cap G = \{0\}$, one can deduce that $F_p \cap N = \{0\}$. Therefore, $F_p \cong q(F_p) \leq (K/N)[p]$, where $q: K\to K/N$ is the canonical quotient homomorphism.
Since $K/N$ is a compact Abelian Lie group, there exists $n \in \N$ such that the $p$-rank of $(K/N)[p]$ is at most $n$ for every prime $p$. Hence, $r_p(F_p) \leq n$ for all $p$, as desired.

 (b) Assume that $|\pi_{(K:G)}| < \infty$ and $(*)$ holds.  Then there is an open \nbd\ $W$ of $0$ in $K$ such that $W \cap K[p] \subseteq G$ for each $p \in \pi_{(K:G)}$.  
By the definition of $\pi_{(K:G)}$, $G$ contains the whole $K[p]$ for $p \notin \pi_{(K:G)}$.  So for any element $g \in W$ of $K$ with prime order $p$, regardless of whether $p$ is in $\pi_{(K:G)}$ or not, we have $\hull{g} \subseteq G$.  Therefore, $W$ witnesses local essentiality of $G$, as $G$ is weakly essential in $K$.

%Let $r_p(K) \ne 0$ for $p=p_1, \ldots, p_s$ and put $m = p_1\cdot \ldots \cdot  p_s$. 
%For $i= 1,\ldots, s$ the subgroup $V_i = G\cap K[p_i]$ of $K[p_i]$ is open. Then $V = \bigoplus_{i=1}^sV_i$ is open in $K[m]=\Soc(K)$
%and $V\leq G$. Hence, there exists an open \nbd\ $W$ of 0 such that $W\cap \Soc(K) = V$. Now every cyclic subgroup $\langle x\rangle$ of prime order contained in $W$ is contained in $\Soc(K)$ as well, hence $\langle x\rangle\leq W\cap \Soc(K) = V \leq G$.  

(c) As $\dim K_1< \infty$, ``some'' and ``every'' are equivalent in (c$_3$), and (c$_3$) is equivalent to local essentiality of $G\cap K_1$ in $K_1$ by Claim \ref{NewClaim*}. 

(c$_3$) $\Rightarrow$ (c$_1$)
Since $G \cap K_1$ is locally essential in $K_1$, it follows that $\Soc(G)$ is locally essential in $\Soc(K)$.
Hence, by Lemma \ref{Neeew}(c), we conclude that $G$ is locally essential in $K$.

(c$_2$) $\Rightarrow$ (c$_3$) 
Let $N$ be a  $\delta$-subgroup of $K_1$. So $N$ is totally disconnected. 
Using Fact \ref{Fact1} one can find a  $\delta$-subgroup $M$ of $K$ containing $N$, so $\pi_{(N:G)} \subseteq \pi_{(M:G)} $ is finite. 
Applying (b) to the subgroup  $G \cap N$ of $N$ (obviously satisfying ($*$)), we deduce that  $G \cap N$ is locally essential in $N$. 

(c$_1$) $\Rightarrow$ (c$_2$) 
The first assertion of (c$_2$)  follows directly from item (a).  
For the other assertion, pick a  $\delta$-subgroup $M$ of $K$. Since $G$ is locally essential, the intersection $G \cap M$ is locally essential in the totally disconnected compact group $M = \prod_p td_p(M)$. It only remains check that $ |\pi_{(M:G)}|   <  \infty$.  Let $W$ be an open neighbourhood of 0 in $K$ witnessing the local essentiality of $G$ in $K$. Since $M=\prod_p td_p(M)$ carries the product topology, $W\cap M$ contains all but finitely many of the components $td_p(M)$. In particular, $W$  contains all but finitely many of the subgroups $M[p] \leq td_p(M)$. Therefore, $\pi_{(M:G)}$ is finite.

The final assertion of (c) is obtained by (c$_2$) taking $M=K$ which is possible when $K$ is totally disconnected
\end{proof}

\section{Locally minimal \sco \ Abelian groups}\label{Sec4}\label{SCO}
%\hfill

In this section we examine the locally minimal Abelian groups under the looking glass of sequential completeness.  In \S \ref{Sec:LocMin:vs:Min} we study the effect of this condition on the connected component of a locally minimal group.  In \S \ref{Sec:LocMin:vs:Min2} we focus on hereditarily disconnected locally minimal groups. While in \S \ref{Sec:LocMin:vs:Min} we single out additional properties ensuring that local minimality coincides with minimality, in \S \ref{Sec:LocMin:vs:Min2} we produce locally minimal \sco \ Abelian groups that are not minimal. 

\subsection{Locally minimal \sco \ Abelian groups: proof of Theorem \ref{IntroA}}\label{Sec:LocMin:vs:Min} \hfill
    
    \medskip
    
%\subsection{Sequentially complete groups}\label{SCO}

%\begin{definition}\label{SecCo} A topological group $G$ is said to be \sco\  if every Cauchy sequence  in $G$ converges.  \end{definition}

Since a topological group $G$ is \sco \ precisely when $G$ is sequentially closed in its  completion $\widetilde G$, 
countably compact groups as well as complete groups are \sco, but sequential completeness is a much weaker property. While countable compactness imposes various restraints on the algebraic properties of the underlying group $G$ when it is infinite (e.g., $|G| \geq \mathfrak c$, $r_p(G)$ is either finite  or $r_p(G)  \geq \mathfrak c$ for every prime $p$, etc. see \cite{DTk1}), every infinite Abelian group $G$ admits a non-discrete group topology which makes it \sco, as the next example shows: 

\begin{example}\label{Exa:Bohr} Every topological group without non-trivial convergent sequences is \sco \ (as every faithfully enumerated Cauchy sequence $(x_n)$ gives rise to a convergent sequence defined by $y_n = x_nx_{n+1}^{-1}$). 

 Therefore, every infinite Abelian group $G$ carries a non-discrete \sco \ group topology -- 
one can consider the group $G^\#$ which is known to have no non-trivial convergent sequences \cite{Flor}. %\NB 
%\NB In contrast, the only complete group topology on a countably infinite group is  
\end{example}

Sequential completeness is preserved under taking direct products and sequentially closed subgroups.  It is not preserved under taking continuous homomorphic images (as every discrete group is \sco). Now we see that it is preserved under taking some quotients. 

\begin{lemma}\label{Le:qsco} Let $G$ be a locally precompact sequentially complete group and $N$ a locally compact normal subgroup. Then the quotient group $G/N$ is sequentially complete.
\end{lemma}

\begin{proof} Let $p: \widetilde{G}\to \widetilde{G}/N$ be the quotient mapping. By the theorem \cite[The Sequence Lifting Theorem]{Varo} of Varopoulos, for every sequence $(x_n)$ in $G/N$  converging to a point $y\in \widetilde{G}/N$, there exists a convergent sequence $(g_n)$ in $G$ with the limit point $h$ such that  $p(g_n)=x_n$ and $p(h)=y$. Since $G$ contains $N$ and $x_n\in G/N$, we have that $p^{-1}(x_n)\subseteq G$; so in particular, that $g_n\in G$.
By sequential completeness of $G$, we obtain that $h\in G$. So $y=p(h)$ is in $G/N$. This implies that $G/N$ is sequentially closed in the locally compact group $\widetilde{G}/N$. So $G/N$ is sequentially complete.
\end{proof}

\begin{proof}[{\bf Proof of Theorem \ref{IntroA}}] We have to prove that if $G$ be a locally precompact, sequentially complete, locally minimal  Abelian group, then 
\begin{itemize}
\item[(a)] $w(c(G))=w(c(\widetilde{G}))$.
\end{itemize}
 Moreover, if $w(c(G))$ is not Ulam-measurable, then 
 \begin{itemize} 
 \item[(b)] $c(G)=c(\widetilde{G})$,  in particular, $c(G)$ is locally compact. 
 \end{itemize}

We start with the following reduction.  Since $\widetilde{G}$ is locally compact, it is of form $\R^n\oplus G_0$, with the product topology,  where $n\in \N$ and $G_0$ has an open compact subgroup $K$ \cite{DPS}. Let us show first that we can assume without loss of generality that $G_0 = K$ is compact.  Indeed, $\R^n \oplus K$ is  an open subgroup of  $\widetilde{G}$, so $G_1:= G \cap (\R^n \oplus K)$ is dense in $\widetilde{G_1} = \R^n \oplus K$.   Moreover, $G_1$ as an open subgroup of $G$, is locally precompact, sequentially complete and locally minimal, with $c(G) = c(G_1)$ and $c(\widetilde{G}) = 
c(\widetilde{G_1})$. Hence, if we can prove the conclusion of the theorem for $G_1$, it will hold for $G$ as well. Therefore, 
we assume from now on that $\widetilde{G} =  \R^n \oplus K$.  

Since the theorem trivially holds when $c(\widetilde{G})$ is trivial, we assume from now on that $w(c(\widetilde{G}))\geq \omega$. We shall first verify the following weaker version of (a):

\begin{itemize} 
\item[(a$'$)] \emph{If $\dim \widetilde{G}$ is infinite, then $w(c(G)) = w(c(\widetilde{G}))$.}
\end{itemize}

We use (a$'$) to derive (b).  Once (b) is established, we deduce the remaining case $\dim \widetilde{G}< \infty$ of (a)
from (b), using the fact that $\dim \widetilde{G}<\infty$ implies $w(c({G}))  \leq w(c(\widetilde{G})) = \omega$. Since 
$\omega$ is not Ulam-measurable, (b) yields $c(G) = c(\widetilde{G})$, in particular, $w(c(G))=w(c(\widetilde{G}))$.  
%In particular, this yields (a).

To prove (a$'$) note that $c(\widetilde{G})=\R^n\oplus c(K)$. Then $\dim \widetilde{G} = \dim K = w(c(K))$ is infinite.  
Since $G$ is locally essential in $\widetilde{G}$, the subgroup $H := G \cap K$ is locally essential in $K$.
It then follows from Proposition~\ref{NEW:proposition}(b) that
$w(c(\overline{H})) = w(c(K)).$

\medskip

By \cite[Corollary 3.3]{DDHXX}, there exists a \nbd\ $U$ of $0$ in $K$ such that every closed subgroup of $H$ contained in $U$ is minimal. Let $N$ be a co-NSS subgroup of $K$ contained in $U$. Then $H_1=H\cap N$, hence $\overline{H_1}$ as well, are locally essential in $N$ by Claim \ref{NewClaim*} . By the choice on $U$, $H_1$ is minimal and \sco, being a closed subgroup of $G$.

Since $\dim K/N = \dim c(K)/c(N)< \infty$  (as $N$ is co-NSS in $K$), 
we conclude that $w(c(N))=w(c(K))$  and $\dim N =  \dim K$ is infinite as well. Then local essentiality of 
$\overline{H_1}$ in $N$ and Proposition~\ref{Prop:lin}(b) imply that 
                           $$w(c(\overline{H_1}))=w(c(N))=w(c(K)).$$

 It follows from Theorem \ref{Thm:DTk} that for $H_1$ in $\overline{H_1}$, we have
                           $$w(c(H_1))=w(c(\overline{H_1}))=w(c(K)),$$
%                           
%=====================================
%
%Moreover, by Proposition~\ref{Prop:lin}(d), the precompact locally minimal group $H$ contains a closed minimal subgroup $G_1$ such that
%$$
%w(c(\overline{G_1})) = w(c(\overline{H})) = w(c(K)).
%$$
%
%Since $G_1$ is minimal and sequentially complete, it follows from Theorem \ref{Thm:DTk} that
%$$
%w(c(G_1)) = w(c(\overline{G_1})) = w(c(K)),
%$$
%
%
%=========================================
%
which implies that $w(c(G)) \geq w(c(H_1)) \geq w(c(K))$. With the trivial inequality $w(c(G)) \leq w(c(K))$, we get $w(c(G)) = w(c(K)) = w(c(\widetilde{G}))$.

\medskip

%Let $U$ be a neighborhood of $0$ in $K$ witnessing the local essentiality of $G \cap K$, and let $N\subseteq U$ be a closed subgroup of $K$ such that $K/N$ is a Lie group, so $\dim K/N$ is finite. It follows from the equality $\dim K=\dim N\cdot \dim K/N$ that $$w(c(N))=\dim N=\dim K=w(c(K)).$$Therefore, $w(c(N)) = w(c(K))$ and $\dim N$ is infinite. 

%Let $G_1 := G \cap N$ and $N_1 := \overline{G_1}$.  Then both $G_1$ and $N_1$ are essential in $N$. Hence, by Proposition~\ref{NEW:proposition}(b), we have
%$$w(c(N_1)) = w(c(N)) = w(c(K)).$$

%It remains to observe that $G_1$ is sequentially complete, essential, and dense in $N_1$.  By Theorem \ref{Crit},  $G_1$ is minimal, with completion $N_1$. In view of \cite[Theorem 3.2]{DTk}, this gives $$ w(c(N_1)) = w(c(G_1)) \leq w(c(G)).$$  
%It then follows from  $$w(c(K)) = w(c(N_1)) \qquad \text{and} \qquad w(c(G)) \leq w(c(\widetilde{G})) = w(c(K))$$ that $w(c(G))=w(c(\widetilde{G}))$.

(b) We first consider the case when $G$ is precompact, so $n=0$ and $\widetilde{G}=K$ is compact. 

Let $X=\widehat{K}$ and $Y=D(X)$ be its divisible hull. Then $|Y| = |X|$ and $Y/X$ is torsion. Let $L:= \widehat{Y}$. Then there is a continuous surjective homomorphism $f:L\to K$, $\ker f \cong \widehat{Y/X}$ is totally disconnected (so, $\dim L = \dim K$)  and $$w(L) = |Y| = |X|= w(K) = w(G).$$
Moreover,  $f(c(L)) = c(K)$ and $r(X/t(X)) = r(Y/t(Y)) $, so we have $$w(c(L))=|Y/t(Y)|=\omega\cdot r(Y/t(Y))=\omega\cdot r(X/t(X))=|X/t(X)|=w(c(K)),$$
where the second equality is due to our assumption that $c(K) \ne \{0\}$, so $X$, and consequently $Y$, are not torsion.  

On the other hand, the dense subgroup $G_1:=  f^{-1}(G)$ of $L$ is locally essential in $L$ by Lemma \ref{Claim1}.  As $L$ is torsion-free, $G_1$ is minimal, by Corollary \ref{New:Corollary}.  To see that $G_1$ is also  \sco\ fix a sequence $(x_n)$ in $G_1$ converging to a point $y\in L$. Since $G$ is sequentially complete and $f(x_n) \to f(y)$, we deduce that $f(y)\in G$. So we have
  $$
  y\in f^{-1}(f(y))\subseteq f^{-1}(G)=G_1,
  $$
ensuring that $G_1$ is sequentially complete. 

By (a$'$), either $w(c(L)) = w(c(K)) = w(c(G))$, or $\dim L = \dim K < \infty$, so $w(c(L)) = w(c(K))$ is countable. In either case, $w(c(L))$ fails to be Ulam-measurable.
 As $c(G_1) \leq c(L)$, allows us  to conclude that $ w(c(G_1))$  is not Ulam-measurable as well, so from Theorem \ref{Thm:DTk}
 %\cite[Theorem 3.2]{DTk} 
 we deduce that the minimal \sco  \ Abelian group $G_1$ contains $c(L)$, and consequently we conclude that $G$ contains $c(K) = f(c(L))$, as desired. 

\medskip

In the general case our 
%\NB\footnote{This formidable proof is yours. I only changed some order and inserted  a new claim, by removing some material that became obsolete after these steps. Now I have nothing to add here any more and I am happy with this proof.} 
hypothesis, jointly with (a$'$) implies that $w(c(\widetilde{G}))$ is not Ulam-measurable. Indeed, the is obvsious, when $\dim c(\widetilde{G}) < \infty$, since then $\dim c(\widetilde{G}) = \omega$ is not Ulam-measurable. Otherwise 
 (a$'$)
 % to deduce that $w(c(G/F))$ is not Ulam-measurable. Indeed, if $\dim c(\widetilde{G})$ is infinite, then  (a$'$) 
gives $w(c(\widetilde{G})) = w(c({G}))$, so our assumption that $w(c(G))$ is not Ulam-measurable yields that $w(c(\widetilde{G}))$ is not Ulam-measurable as well. 

Let $q: \R^n\oplus K\to \R^n$ be the projection. Since $q(G)$ is dense in $\R^n$, there are elements $g_1, g_2, \dots, g_n\in G$ such that $\{q(g_i): 1\leq i\leq n\}$ is linearly independent over $\R$. Then they generate a discrete free Abelian subgroup of $\R^n$. Therefore, the subgroup $F =\langle g_1, g_2, \dots, g_n\rangle $ of $G$ is also free Abelian and discrete. Hence, $F\cap K=\{0\}$.

Let $\varphi: \widetilde{G}\to \widetilde{G}/F$ be the quotient homomorphism. First of all, note that  the group $\widetilde{G}/F$ is compact, since 
\begin{equation}\label{Eq:June29}
(\widetilde{G}/F)/\varphi(K)\cong \widetilde{G}/(F+K)\cong \T^n
\end{equation}
and $\varphi(K)$ is compact. So $\varphi(G) = G/F$ is precompact. 

On the other hand, $G\cap K$ is locally essential in $K$ and $\varphi$ maps $K$ isomorphically onto $\varphi(K)$, so the subgroup $\varphi(G\cap K)$ is locally essential in $\varphi(K)$. As $\varphi(G\cap K) \leq \varphi(G)\cap \varphi(K)$, the subgroup $\varphi(G)\cap \varphi(K)$  is locally essential in $\varphi(K)$ as well. 
By \eqref{Eq:June29}, $\varphi(K)$ is a co-NSS subgroup of $\widetilde{G}/F$. In view of Claim \ref{NewClaim*}, $G/F=\varphi(G)$ is locally essential in the compact group $\widetilde{G}/F$, therefore is locally minimal. Moreover, $G/F$ is sequentially complete by Lemma \ref{Le:qsco}. Since $G/F$ is precompact (with completion $\widetilde{G}/F$), our next aim will be to show that 

\begin{claim*} $w(c(\widetilde{G}/F)) \leq w(c(\widetilde{G}))$, so $w(c(\widetilde{G}/F))$ and $w(c(G/F))$ are not Ulam-measurable.
\end{claim*}

With this claim we can apply the precompact case to conclude that $c(G/F)$ contains $c(\widetilde{G}/F)$. This will allow us to 
end the proof of (b) by concluding that $G$ contains $c(\widetilde{G})$ since $\varphi(c(\widetilde{G}))\leq c(\widetilde{G}/F) \leq G/F$.

\noindent {\em Proof of the claim.}
Let $H=\overline{c(\widetilde{G})+F}$, then $(\widetilde{G}/F)/(H/F)\cong \widetilde{G}/H$
is totally disconnected being isomorphic to a quotient of the totally disconnected locally compact group $\widetilde{G}/c(\widetilde{G})$.
So $H/F$ contains $c(\widetilde{G}/F)$.
On the other hand, $H/F=\varphi(H)$ is connected since it contains the dense connected subgroup $\varphi(c(\widetilde{G}))=\varphi(c(\widetilde{G})+F)$.
Consequently, we have 
\begin{equation}\label{Eq23}
H/F=\varphi(H) = c(\widetilde{G}/F).
\end{equation}

Since $c(\widetilde{G}) = \R^n \oplus c(K) $,  the group $\widetilde{G}/ c(\widetilde{G}) \cong K/c(K)$ is compact and totally disconnected. Hence, so is its closed subgroup $H_1:= H/c(\widetilde{G})$, which is in addition topologically generated by $n$ elements. Thus, the discrete group $\widehat{H_1}$ is (algebraically) isomorphic to a torsion subgroup of $\T^n$. It follows that $\widehat{H_1}$ is countable, and hence $w(H_1) = w(H/c(\widetilde{G})) \leq \omega$. This proves the equality $w(H)=w(c(\widetilde{G}))$, since $c(\widetilde{G})\ne \{0\}$. Since  $c(\widetilde{G}/F)= H/F$ is a quotient of $H$, we conclude that 
$w(c(\widetilde{G}/F)) \leq w(c(\widetilde{G}))$. As $c({G}/F)$ is isomorphic to a subgroup of $c(\widetilde{G}/F)$, it follows that $w(c({G}/F))$ is not Ulam-measurable.
\end{proof}

 Item (a$'$) remains true for precompact, sequentially complete, locally minimal nilpotent groups (see \S 6 for a proof). 

 As an application of Theorem \ref{IntroA} one obtains a strikingly positive  answer of our initial quest on whether local minimality and minimality coincide for some classes of \sco\ precompact Abelian groups -- this is the case of connected groups of non-measurable weight: 
 
\begin{corollary}\label{Coro:Oct24} For a \sco\ connected precompact Abelian group $G$ such that $w(G)$ is not Ulam-measurable 
the following are equivalent: 
\begin{itemize}

\item[(a)] $G$ is locally minimal; 

\item[(b)] $G$ is minimal; 

\item[(c)] $G$ is compact. 

\end{itemize}
\end{corollary}

This corollary, as well as Theorem \ref{IntroA}, strengthen \cite[Corollary 3.3]{DTk} (resp., \cite[Theorem 3.2]{DTk}) where the same conclusion as in Corollary \ref{Coro:Oct24} (resp., Theorem \ref{IntroA}) is obtained under the stronger assumption that $G$ is minimal. 

\begin{remark}\label{Rem:Oct24} 
For every $n\in\N_+$ the group $G=\R^n$ is connected, locally minimal and complete Abelian group with $w(\R^n) =\omega$ non-measurable, yet $G$ is not compact. Hence ``precompact"  cannot be relaxed in the above theorem and Corollary \ref{Coro:Oct24}.
\end{remark}

\subsection{Hereditarily disconnected locally minimal \sco \ Abelian groups}\label{Sec:LocMin:vs:Min2} \hfill

It is known (\cite[Corollary 4.9]{DTk}) that if $G$ is a minimal hereditarily disconnected  (i.e., with $c(G)=\{0\}$) \sco \ Abelian group, then $c(\widetilde G) =\{0\}$
(or equivalently, $\dim \widetilde G = 0$). The following immediate consequence of the above theorem (namely, item (b) of Corollary \ref{CoroA}) shows that  this property remains true after relaxing ``minimality" to ``local minimality" in conjunction with local precompactness (precompactness is automatically present in the case of minimality due to Prodanov-Stoyanov's Theorem): 

\begin{corollary}\label{Coro:Oct23} 
 If $G$ is a locally minimal, hereditarily disconnected  \sco \ and  locally precompact Abelian group, then $c(\widetilde G) = \{0\}$
 $($or equivalently, $\dim \widetilde G = 0)$. Consequently, $G$ has linear topology.
\end{corollary}

\begin{proof} Since the weight of $c(G) = \{0\}$ is not Ulam measurable, the equality $ c(\widetilde G) = c(G) =\{0\}$ is granted by Theorem \ref{IntroA}(b).  The last assertion as well as   $\dim \widetilde G = 0$ follows from the fact that by Dantzig Theorem the hereditarily disconnected locally compact groups have linear topology, so they are zero-dimensional.
\end{proof}

That is why we consider in the sequel dense locally minimal \sco \ subgroups of the totally disconnected compact Abelian groups. The first restraint for such subgroups is
the counterpart of \cite[Theorem 2.11]{DU} (see Theorem \ref{Products}(a) below).

 We first introduce the next property, which is a counterpart of \cite[Lemma 2.14]{DU}: 

\begin{proposition}\label{tot:disco+cc+loc:min} Let $p$ be a prime and let $A$ be a pro-$p$-finite Abelian group. If $G$ is a dense \sco \ locally minimal subgroup of $A$, then $p^kA \leq G$ for some $k \in \N$ and $G/p^{k+1}A$ is locally minimal.
\end{proposition}   

\begin{proof} Since $A$ (and consequently $G$) have linear topology, $G$ has an open minimal subgroup $G_1$, by Proposition \ref{Prop:lin} (b). 
Since the closure $B$ of $G_1$ is still a pro-$p$-finite Abelian group, \cite[Lemma 2.14]{DU}, entails that $p^l B \subseteq G_1 \leq G$
for some $l \in \N$. On the other hand, $p^m A \subseteq B$ for some $m \in \N$ since $B$ is an open subgroup of $A$. Therefore, $p^{k} A \leq  G$
for $k = l+m$. 
%%%%%%%%%%%%%
%By  It follows from the definition of pro-$p$-finite groups, that  the Pontryagin dual $X$ of $A$ is a discrete $p$-group with $\sigma  := |X| = w(A)$. Then 
%$X$ is isomorphic to a subgroup of $\Z(p^\infty)^{(\sigma)}$. Taking once again the duals we deduce that   there exists a continuous surjective homomorphism $f: \Z_p^\sigma \to A$. Arguing as in the proof of \cite[Lemma 2.15]{DU} apply the above lemma to deduce that the subgroup  $G_1= f^{-1}(G)$ is a dense locally minimal \sco \ subgroup of $\Z_p^\sigma$. Since $\Z_p^\sigma$ is torsion-free, this implies that $G_1$ is minimal and countably compact. By \cite[Lemma 2.15]{DU}, $p^k\Z_p^\sigma \leq G_1$ for some $k \in \N$. Therefore, $p^kA \leq G$.

The second assertion, that $G/p^{k+1}G$ is locally minimal follows directly from the following claim which  completes proof of the proposition:
\begin{claim*} Let $G$ be a subgroup of a pro-$p$-finite Abelian group $A$ containing $p^kA$ for some $k\in \N$. If $G$ is locally essential in $A$, then $G/p^{k+1}A$  locally essential in $A/p^{k+1}A$.
\end{claim*}
\noindent{\em Proof of the claim.}  Fix an open subgroup $V$ of 0 in $A$ witnessing local essentiality of $G$, then $V':= \Soc(A) \cap V \leq G$ is an open (hence, finite index) subgroup of $\Soc(A)$. Due to Lemma \ref{Rem:29} it is enough to check that under the quotient map $q: A \to A/p^{k+1}A=:A^*$ there is an open subgroup $W$ of $A^*$ such that $W \cap \Soc(A^*) \leq q(G)$. 

Since $\Soc(A^*) = (p^kA + \Soc(A))/p^{k+1}A$,  $p^kA \leq G$ and $V' \leq G$, the (open, so finite-index ) subgroup $B= p^kA + V'$ of $p^kA + \Soc(A)$ is contained in $G$. Thus $q(B) \leq q(G)$ and $q(B)$ is an open subgroup of $\Soc(A^*)$. Then any open subgroup $W$ of $A^*$ with $W \cap  \Soc(A^*) = q(B)$ will be as desired. 
%This proves that $Soc(q(G))$ is locally essential in $A^*$, as $\pi(A^*) = \emptyset$. $q(B) \cap \Soc(A^*) $ is open. Hence, we can apply Theorem \NB\footnote{was Lemma}  \ref{claim4} to deduce that $q(G)$ is locally essential in $A^*$.
\end{proof}
%To check that $G/p^{k+1}A$ is locally minimal fix an open subgroup $V$ of 0 in $A$ such that $V':= \Soc(A) \cap V \leq G$. Since $V$ is open, $V'$ is an open subgroup of $\Soc(A)$, hence of finite index. Due to Remark \ref{Rem:29} it is enough to check that under the quotient map $q: A \to A/p^{k+1}A=:A^*$ there is an open subgroup $W$ of $A^*$ such that $W \cap \Soc(A^*) \leq q(G)$. 
%
%Note that $\Soc(A^*) = (p^kA + \Soc(A))/p^{k+1}A$. Since $p^kA \leq G$ and $V' \leq G$, we have a finite-index (open) subgroup of $p^kA + \Soc(A)$, namely $B= p^kA + V'$, contained in $G$.  Hence, $q(B) \leq q(G)$ and $q(B) \leq \Soc(A^*) $ is a finite-index open subgroup of $\Soc(A^*)$. 
%This proves that $q(B) \cap \Soc(A^*) $ is open. Hence, we can apply Lemma \ref{claim4} to deduce that $q(G)$ is locally essential in $A^*$. 
%Hence, there exists an open \nbd \ $W$ of 0 in $A^*$ with $W \cap \Soc(A^*) \subseteq q(B)$. Since $A$ has linear topology it is not restrictive to assume that 
%$W$ is an open subgroup of $A^*$. This will be the desired $W$, which proves that 
%\end{proof} 

%%%%%%%%%%%%%%%%%%%%%%%%%%%%%%%%%%%%%
% - - - - - - - - - - important for Powers - - - - - - - - - - - 
%
%\NB\footnote{This corollary will be used in the other (Powers) project to see that such a $G$, 
%$G^\kappa$ is locally minimal for every $\kappa$ iff $G$ is minimal. } 
%Combining Theorem \ref{Products} and Proposition \ref{tot:disco+cc+loc:min} 
We obtain now a description of the  \sco\ locally \mi \ totally disconnected precompact Abelian group as extensions of a compact groups by a 
 locally minimal  \sco\ group that is a product of bounded $p$-groups. 

\begin{theorem}\label{Products} Let $G$ be a \sco\ locally \mi \ totally disconnected precompact Abelian group with $K=\widetilde G$. 
Then 
\begin{itemize}
\item[(a)] $K$ is totally disconnected, so $K = \prod _{p\in \Prm} K_p$ and $G= \prod _{p\in \Prm} G_p$, where $G_p=td_p(G)$ and $K_p=td_p(K)$ for every $p\in \Prm$; 

\item[(b)] for every $p\in \Prm$ there exists $k_p\in \N$, such that $N= \prod_pp^{k_p}K_p\leq G$ and, with $N^*= \prod_pp^{k_p+1}K_p$, the quotient $G/N^*$ is locally minimal. 
\end{itemize}
\end{theorem}

\begin{proof} (a) The compact completion $K$ of $G$ is totally disconnected by 
Corollary \ref{Coro:Oct23}. Therefore, $K=\prod_pK_p$, where $K_p=td_p(K)$ is closed for every for every $p\in \Prm$.
Hence, for each $p\in \Prm$ the subgroup $td_p(G)=G\cap K_p$ of $G$ is closed. 
The verification of the conclusion $G\cong \prod _{p\in \Prm} td_p(G)$ uses only the fact that $G$ is \sco, 
and remains the same as in the proof of \cite[Theorem 2.11]{DU}.

(b) Proposition \ref{tot:disco+cc+loc:min}  provides  for every $p\in \Prm$ a $k_p\in \N$ such that $p^{k_p}K_p\leq G$ and $G_p/p^{k+1}K_p$
is locally minimal. Then 
$\bigoplus_p p^{k_p}K_p\leq G$. Since the closure $N$ of this sum coincides also with its sequential closure, 
the inclusion $N= \prod_pp^{k_p}K_p\leq G$ follows from the fact that 
$G$ is \sco. Finally, $G/N^* = \prod_p G_p/p^{k_p+1}K_p$. Since for every prime $p$ the compact group $K_p/p^{k_p+1}K_p$
is a bounded $p$-group, the subgroup $G/N^*$ of  the compact group $K/N^*$ is weakly essential. Moreover, 
the local essentiality of $G$ in $K$ implies that $G$ is essential in an open subgroup of $K$.
In particular, there exists a finite subset $P \subseteq \P$, such that $G_p \leq K_p$ is essential (so, $K[p] \leq G$ ) for every $p \in\P \setminus P$. Combined with $p^{k_p}K_p\leq G$, this proves that 
$$
(K_p/p^{k_p+1}K_p) [p] = (K/N^*)[p] \leq (G_p/p^{k_p+1}K_p)[p] \leq G/N^* \ \mbox{ for every } \ p \in\P \setminus P. 
$$
  Since local minimality of $G_p/p^{k_p+1}K_p$ implies that $(G_p/p^{k_p+1}K_p)[p]$ is an open subgroup of $(K_p/p^{k_p+1}K_p)[p] = (K/N^*)[p]$ for every $p \in P$, we conclude that $G/N^*$ is locally esssential in $K/N^*$, hence locally minimal, by Theorem \ref{Crit}. 
\end{proof} 

If ``precompact'' is weakened to be ``locally precompact'' in the item (a) of the above theorem, then $G$ would have an open subgroup with the properties.
The above theorem also reduces the study of \sco\ locally \mi, hereditarily disconnected, precompact Abelian groups to the case where their completions are pro-$p$ groups.

In the chaise of locally minimal non-minimal \cc \ Abelian group one can  exploit Proposition \ref{Prop:lin} trying to arrange for a minimal \cc \ Abelian group
$G_1$ with completion $K_1$ (so that, necessarily, $G_1 \geq \Soc(K_1)$). Then a finite extension $G$ of $G_1$ can be arranged, so that $G_1$ is an open finite index
subgroup of $G$ and $K:= K_1 + G$ is compact with $\Soc(K) \not \leq G_1$. 

In the next example we show that this plan works and \cc \ can even be strengthen to $\omega$-boundedness. The idea to use bounded groups is motivated by the tendency to stay as far as possible from the field covered by Proposition \ref{tot:disco+cc+loc:min} (indeed, for bounded groups that proposition becomes vacuous). 
 
 \begin{example}\label{EXAMPLE} There exists an $\omega$-bounded, non-minimal, locally minimal Abelian group. Let $p$ be a prime number and $\tau$ an uncountable cardinal. We set $K=\Z(p^2)^\tau$ and consider the subgroup $$H=\{f\in K: f(\alpha)\in \Z(p)~\mbox{~for~all~but~countably~many~}\alpha<\tau\}.$$
To see that $H$ is $\omega$-bounded we note that for every countable subset $A$ of $H$, there exists a countable subset $E$ of $\tau$ such that 
$$A\subseteq P_E:=\{f\in K: f(\alpha)\in \Z(p)~\mbox{for~all~}\alpha\in \tau\setminus E\},$$ and $P_E$ is a compact subgroup of $H$.
 Moreover, it is evident that $H$ is minimal by the minimality criterion because it contains $pK$ and $pK$ is essential in $K$.
 Take $x\in K\setminus H$ and let $G=H+\hull{x}$. Then $H$ is a dense subgroup of $G$ of index $p$.
 We now refine the topology on $G$ by letting $H$ be open; this refinement is proper, so $G$ is not minimal with this new topology.
 Moreover, as the minimal group $H$ is open in $G$, $G$ is locally minimal and $G$ is also $\omega$-bounded because $G$ is homeomorphic to $H\times \Z(p)$.
 \end{example}

For the stronger property ``minimal" in place of ``locally minimal'', the final conclusion of item (b) of Theorem \ref{Products}  was proved without the restriction ``totally disconnected"  in \cite{DU} (it was proved that $G$ contains a compact subgroup $N^*$ such that $G/N^*$  is locally minimal and a product of finite cyclic groups). 
Let us see that a \sco\ critical locally minimal   Abelian group $G$ fails to have this property. Indeed, if $G$ contains such a compact subgroup $N^*$, then clearly $N^*$ contains $c(G)$ which must be compact. 
Therefore, the quotient group $G/c(G)$ contains a compact subgroup $N^*/c(G)$ such that $(G/c(G))/(N^*/c(G)) \cong G/N^*$ is locally minimal. 
%\NB\footnote{A wrong reference, Remark \ref{New:Remark} has nothing to do here. We need to cite a (three-space) theorem (w.r.t. a compact subgroup) from \cite{DDHXX1}.} Remark \ref{New:Remark}, 
This would imply that $G/c(G)$ is locally minimal contrary to our assumption that $G$ is critical \cite{DDHXX1}. 

The following proposition recovers and reinforces the conclusion of \cite[Corollary 2.11]{DU}, whose proof, based on  \cite[Proposition 2.10]{DU} contains a gap at the very last stage. The new proof is simply a combination of Lemma \ref{Le:qsco}, Fact \ref{New:Fact} and Corollary \ref{Coro:Oct23}.

\begin{proposition}\label{coro:tor_free} If $G$ is a \sco\ minimal   Abelian group with compact 
 $c(G)$, then  $G/c(G)$ is \sco, minimal and has linear topology.
\end{proposition}

This proposition has an immediate corollary:

\begin{corollary} 
A \sco \ critical locally minimal Abelian group fails to be minimal.
\end{corollary}

Critical locally minimal Abelian group with even stronger properties (namely, $\omega$-bounded) will be built in the next section.  
%The next theorem will provide a result in the opposite direction, namely  locally minimal and minimality coincide for connected \sco\ precompact Abelian groups non-measurable weight (see Corollary \ref{Coro:Oct24}). 

\section{Sequentially complete critical locally minimal Abelian groups }\label{Sec:Clm}

In this section we face the following question triggered by Proposition \ref{coro:tor_free}: 

\begin{question}\label{ques:coro:tor_free} If $G$ is a \sco\ locally minimal   Abelian group with $c(G)=c(\widetilde G)$, is then also $G/c(G)$ locally  minimal?
\end{question}

In other words, reformulating the question in a counter-positive form: if $G$ is a \sco\ locally minimal precompact Abelian group with $c(G)=c(\widetilde G)$ can 
$G$ be a critical locally minimal, i.e.,  is the class $\CCC$ of the completions of these groups non-empty? Proposition \ref{PropSuff} below gives a negative answer to Question \ref{ques:coro:tor_free} and provides a proof of Theorem \ref{IntroB}.

\subsection{Sufficient conditions for $K \not \in \CCC$}
\hfill

 The next theorem collects various instances when the answer to Question \ref{ques:coro:tor_free} is positive.

Now we prove that if $K$ is a compact Abelian group such that $C$ is a hybrid torus or splits topologically, then $K \not \in \CCC$, i.e., 
 $K$ contains no dense critical locally minimal groups. The same occurs also when $\overline{\Soc(K)}$ is totally disconnected. 
% for every dense locally minimal subgroup $G$ with $c(G)=c(K)\cong C$, the quotient $G/c(G)$ is locally minimal. Now we extend this to groups $K$ such that $c(K)$ is a 

\begin{theorem}\label{coro:tor_free**} If $K$ is a compact Abelian group satisfying some of the following conditions (a)--(d), then $K \not \in \CCC$:
\begin{itemize}
   \item[(a)]  $\overline{\Soc(K)}$ is totally disconnected;
   \item[(b)] the set 
   %\NB\footnote{was $\pi^*(K)$} 
   $\pi_{tor}(K)$ is finite;
   \item[(c)] when $c(K)$ splits topologically, in particular, when $c(K)$ is a torus;
   \item[(d)] when $c(K)$ is a hybrid torus.
\end{itemize}
 \end{theorem}

\begin{proof} Let $q: K \to K/c(K)$ be the canonical homomorphism and put $K_1:= \overline{\Soc(K/c(K))}$. 

(a) In case $K_1$ is totally disconnected, since $q(\Soc(K))=\Soc(K/c(K))$ by Claim \ref{LAST:claim}, $\overline{\Soc(K/c(K))} = q(K_1)$ is totally disconnected as well. 
Let $G$ be a dense locally minimal subgroup of $K$ with $c(G) = c(K)$. It suffices to prove that $G/c(K)$ is locally minimal. By Claim \ref{LAST:claim} again, 

(i) if $G$ satisfies ($*$) from Theorem \ref{claim4}(b) then also $q(G)$ does so;

(ii) if $G$ is weakly essential in $K$ then also $q(G)$  is weakly essential in $K/c(K)$. 
 
 Since $K_1$ is totally disconnected, local essentiality of $G$ in $K$ implies that $G\cap K_1$ is locally essential in $K_1$. Therefore, 
 $\pi_{(K:G)} = \pi_{(K_1:G)}$ is finite, by the final assertion of Theorem \ref{claim4}(c). Since $\pi_{(K:G)} \supseteq \pi_{(q(K):q(G))}$, this implies that $\pi_{(q(K):q(G))}$ is finite as well. Hence, Theorem \ref{claim4} (b) implies that $q(G)$ is locally essential in $K/c(K)$. By Theorem \ref{Crit}, $q(G)$ is locally minimal. 
%Hence, there exists an open \nbd\ $W$ of 0 such that $W\cap \Soc(K) = V$. Now every cyclic subgroup $\langle x\rangle$ of prime order contained in $W$ is contained in $\Soc(K)$ as well, hence $\langle x\rangle\leq W\cap \Soc(K) = V \leq G$.

(b) Let $P=\pi_{tor}(K)$ and $m = \prod_{p\in P}p$. In view of (a), it is enough to check that $K_1$ is totally disconnected. Indeed,  $Soc(K) = \bigoplus_{p\in P}K[p] = K[m]$ is compact. Hence, $K_1 = Soc(K)$ has finite exponent, so it is  totally disconnected. 
%
%%%%%%%%%%% here lies the second proof that nobody wants, amen  %%%%%%%%%%%%%%%%%%%%
% \NB\footnote{ Let us give now a second, direct proof in this special case. . . . \\
% Should we give this second (but much longer) proof ? \BH I agree to remove it
%% %Assume that $r_p(K) \ne 0$ only for $p=p_1, \ldots, p_s$. By Theorem \ref{claim4}(a), ($*$) from that theorem  holds. This means that $V_i=G\cap K[p_i]$ is an open subgroup of $K[p_i]$ for all $i = 1,\ldots, s$. Put $K_1 = K/c(K)$ and let $q: K \to K_1$ be the canonical homomorphism. By Claim \ref{LAST:claim}, $q(K[p]) = K_1[p]$. 
%%Moreover, $q(V_i)$ is an open subgroup of $K_1[p_i]$ for   $i = 1,\ldots, s$. As $V_i \leq G$, we have $q(V_i) \leq q(G) = G/c(G)$. This shows that 
%%$K_1$ and $q(G)$ satisfy ($*$) from Theorem \ref{claim4}(b). Therefore, $q(G)$ is locally essential in $K_1$.
%}

(c) Assume that $c(K)$  splits topologically. Then $K=c(K)\oplus D$ for some necessarily totally disconnected closed subgroup $D$. As $G$ contains the factor $c(K)$, $G=c(K)\oplus G_1$, where $G_1=G\cap D$ is dense in $D$. So $G/c(G)$ is topologically isomorphic to the closed subgroup $G_1$ of $G$, which must be locally minimal.

(d) Put $C = c(K)$, then $\dim C < \infty$ as $C$ is a hybrid torus. Assume that $G$ is a dense locally minimal subgroup of $K$ with $c(G) =C$. We need to prove that $G/C$ is locally minimal. 

Fix a $\delta$-subgroup $H$ of $K$.  Then $\pi^*(H)$ is finite, by Proposition \ref{prop:ht}.
 Let $q_1: K \to K/H$ be the quotient homomorphism.  Then $c(K/H) = q(C) = C/H$ is a torus, so by Fact \ref{torisplit} the group $K/H$ admits a totally disconnected closed subgroup $D'$ such that
$$
K/H = D' \oplus (C/H),
$$
where the decomposition is topological. Since $q(G)$ contains $C/H$, we may write $q(G) = E'\oplus (C/H)$ for some dense subgroup $E'$ of $D'$.  
Then $E := q_1^{-1}(E')$ is a dense subgroup of $D := q_1^{-1}(D')$.  Moreover, $K = D + C$ and $G = E + C$.

The composition of isomorphisms
$$
K/C \cong (K/H)/(C/H) \cong D',
$$
map $G/C$ onto $E'$, so we identify $K/C$ with $D'$ and $G/C$ with $E'$ in the sequel.

Since $G$ is a dense locally minimal subgroup of $K$, it is locally essential, and in particular, weakly essential in $K$. By Lemma~\ref{Rem:29}, the group $E'$ is weakly essential in $D' \cong K/C$.  In view of Theorem \ref{Crit}, to prove that $G/C = E'$ is locally minimal, it suffices to check that $E'$
is locally essential in $D'$. By Theorem~\ref{claim4}(b) it is enough to verify the following:
\begin{itemize}
    \item[(i)] The set $Q':=\pi_{(D':E')}$ is finite;
    \item[(ii)] For each prime $p$, the subgroup $D'[p] \cap E'$ is open in $D'[p]$.
\end{itemize}

Since $E$ is a closed subgroup of the locally minimal group $G$, it is itself locally minimal, and thus locally essential in $D$. By Theorem~\ref{claim4}(c), the set $Q:=\pi_{(D:E)}$ is finite, as the compact group $D$ is totally disconnected. 

To prove (i), it suffices to prove that $Q' \subseteq \pi^*(H) \cup Q$, in other words, if $p \notin \pi^*(H) \cup Q$, then $p \notin Q'$. By Claim \ref{LAST:claim}(c) applied to $D \to D' = D/H$ and our hypothesis $p \notin \pi^*(H)$ we deduce that $q(D[p]) = D'[p]$. The part $p \notin Q$ of our hypothesis gives $D[p]\leq E$.
Therefore, $D'[p] = q(D[p]) \leq q(E) = E'$, i.e., $p \notin Q'$. This finishes the proof of (i). 

%Let $p$ be a prime and let $N' \cong \Z(p)$ be a subgroup of $D'$. Then there exists a closed $p$-monothetic subgroup $N$ of $D$ with $q(N) = N'$. If $p \notin \pi'(H) \cup Q$, then $N \cap H = \{0\}$, implying $N \cong N' \cong \Z(p)$ and hence $N \leq E$.  Thus, $N' = q(N) \leq q(E) = E'.$ So we conclude that $Q' \subseteq \pi'(H) \cup Q$ is finite.

To verify (ii), consider the quotient map $\varphi: K \to D'$ with kernel $C$.

By Claim~\ref{LAST:claim}, we have $D'[p] = \varphi(K[p])$ for any prime $p$. Since $G$ contains an open subgroup of $K[p]$, it follows that $E' = \varphi(G)$ contains an open subgroup of $D'[p]$.  Therefore, $D'[p] \cap E'$ is open in $D'[p]$.
 \end{proof}

\begin{corollary}\label{MainCoro} If $G$ is a critical locally minimal Abelian group such that $c(G)$ is finite-dimensional, then $\pi^*(N)$ is infinite for every $\delta$-subgroup $N$ of $c(G)$.
\end{corollary}

\begin{proof} Let $K$ be the completion of $G$ and let $N$ be $\delta$-subgroup of $C:=c(K)=c(G) $. By Proposition \ref{prop:ht},  
$\pi^*(N)$ must be infinite, as $c(K)$ cannot be a hybrid torus by the  above theorem (as far as $K$ contains the dense  critical locally minimal subgroup $G$).  
\end{proof}

\subsection{Sufficient conditions for $K \in \CCC$}
\hfill

Here we provide a counterpart of Theorem \ref{coro:tor_free**}, giving sufficient conditions for $K \in \CCC$, when $K$ is a compact Abelian group.  
 
 \begin{definition}
Denote by $\KK$ the  class of compact Abelian groups $K$  with locally essential $c(K)$.
\end{definition}

In order to provide sufficient conditions for $K \in \CCC$, we first provide sufficient conditions for $K \in \KK$ and then we characterize in Theorem \ref{Thm:New} the groups $K\in \KK$ that belong to $\CCC$. 

  We saw in Corollary \ref{NEW:corollary} 
that $c(K)$ cannot be essential in $K$ unless $c(K)=K$. While obviously $K \in \KK$ whenever $c(K)$ is open in $K$ (in such a case $K\cong c(K)\times K/c(K)$ splits topologically; it is easy to see that whenever $c(K)$ splits topologically in $K$, $K \in \KK$ precisely when $c(K)$ is open). We shall see now 
that $\KK$ contains a wealth of non-connected groups.  

\begin{lemma}\label{D:claim}
If $K \in \KK$ with $C=c(K)$, then: 

\begin{itemize}
  \item[(a)]  there exists $n \in \N$ such that  for every $p\in \P$, $td_p(K/C)$ is a finite $p$-group with $r_p(td_p(K/C))\leq n$; 
  \item[(b)] $K / C$ is metrizable and topologically generated by $n$ elements;
  \item[(c)] if $w(K) \leq \mathfrak{c}$, then also $K$ is topologically finitely generated.
\end{itemize}
\end{lemma}

\begin{proof} (a) Since $C$ is locally essential in $K$, it follows from Lemma~\ref{Rem:29} that the trivial group is weakly essential in $K/C$, hence $K/C$ is a totally disconnected exotic torus.
By Theorem~\ref{Thm:ex:tori}, each $td_p(K/C)$ is a bounded $p$-group. To evaluate $r_p(K/C)$ use the canonical  homomorphism $q: K \to K/C$. 
Since $C$ is connected, Claim~\ref{LAST:claim} gives  
\begin{equation}\label{Eq:July4}
(K/C)[p] = q(K[p]) \cong K[p]/C[p],
\end{equation}
for every $p\in \P$. By Theorem~\ref{claim4}(a), there exists $n \in \N$ such that 
(in view of \eqref{Eq:July4}) $r_p(K/C)  = r_p(K[p]/C[p]) \leq n$ for each $p\in \P$. 

(b) Since $td_p(K/C)$ is a finite $p$-group of $p$-rank $\leq n$ for each $p\in \P$, $td_p(K/C)$ is a quotient of $\Z_p^n$. Therefore, $K/C$ is a quotient of $(\prod_{p \in \P} \Z_p)^n$, so metrizable. 
Since $\prod_{p \in \P} \Z_p$ is monothetic, $K/C$ is topologically generated by $n$ elements.

(c) Since $C$ is connected and $w(C) \leq \mathfrak{c}$, $C$ is monothetic (see \cite[Corollary 14.1.4]{ADGB}). 
%its Pontryagin dual group $\widehat{C}$ admits an injective homomorphism into the 1-dimensional torus $\T$. Dually, there exists a homomorphism $\Z \to C$ with dense image; that is, 
As $C$ is monothetic and $K/C$ is topologically finitely generated, it follows that $K$ is also topologically finitely generated.
\end{proof}

\begin{remark}\label{New:Remark}
The above implication cannot be inverted as easy examples show (e.g., for any connected $C$,  and  $K = C \times \prod_p\Z(p)$,
the group $K/C$ is even monothetic, but $K\not \in \KK$). \end{remark}

If $K$ is finite-dimensional one can find a stronger property of the pair $K, C$ that characterizes $K\in \KK$. Namely, if $N$ is a $\delta$-subgroup of $K$, then $L:= N \cap C$ is a $\delta$-subgroup of $C$, so for every prime $p$ there is a bound on the $p$-rank of $L_p=td_p(L)$ -- it must be an at most $n$-generated $\Z_p$-module, where $n=\dim C = \dim K$.

 \begin{theorem}\label{inKK} For a compact Abelian group $K$ with $\dim K<\infty$ and $C = c(K)$, the following are equivalent:
 \begin{itemize}
  \item[(a)] $K \in \KK$; 
  \item[(b)] there is a $\delta$-subgroup $N$ of $K$ such that $L:=N\cap C$  is locally essential in $N$; 
  \item[(c)] there exists a $\delta$-subgroup $H$ of $C$, a totally disconnected compact Abelian group $N$ and a closed essential subgroup $L$ of $N$ such that: 
\begin{itemize}
  \item[(c$_1$)] there exists a topological isomorphism $j: L \to H$;
  \item[(c$_2$)] the quotient group $K'= (N\times C)/\Gamma_j$, where $\Gamma_j$ is the graph of $j$, is isomorphic to an open subgroup of $K$. 
\end{itemize}
\end{itemize}
 \end{theorem}
  
  \begin{proof} (a) $\Rightarrow$ (b).  According to Fact \ref{Fact0}(b)),  $K$ admits a $\delta$-subgroup $N$. Since $C$ is locally essential in $K$ and $N$ is closed, we conclude that $L:=N\cap C$  is locally essential in $N$.
  
  (b) $\Rightarrow $ (c). Let $N$ be as in~(b). Then, by Fact~\ref{Fact1}(c) and~(f), we have $K = N + C$ and 
$L := N \cap C$ is a $\delta$-subgroup $H$ of $C$.

Since $L$ is locally essential in $N$ and $N$ is totally disconnected, there exists an open subgroup $N' \leq N$ such that $L' := L \cap N'$ is essential in $N'$. By Fact \ref{Fact1}(g$_4$), the open subgroup $L'$  of $L$ is a $\delta$-subgroup of $C$.

From the equality $K = N + C$, we obtain a quotient homomorphism
$$
\varphi : N \times C \to K, \quad (x, y) \mapsto x + y.
$$
It follows that $K' := N' + C = \varphi(N' \times C)$ is open in $K$.
Let $\varphi'$ denote the restriction of $\varphi$ to $N' \times C$.

Then the kernel of $\varphi'$ is
$$
\ker \varphi' = \{(x, -x) : x \in L'\} = \Gamma_j,
$$
where $j$ is the topological isomorphism
$$
j : L' \to L', \quad x \mapsto -x.
$$
Hence, we have
$$
K' = (N' \times C)/\ker \varphi' = (N' \times C)/\Gamma_j.
$$

 (c) $\Rightarrow $ (a). %Since  the $\delta$-subgroup $H$ of $C$ is locally essential, $\dim C < \infty$, by Corollary \ref{Cor:13May}. 
The subgroup $K'$ in (c$_2$) is isomorphic to an open subgroup of $K$ with $c(K')=c(K)$, so $K' \in \KK$ implies 
  $K \in \KK$. To check that $K' \in \KK$ note that $c(K') = \phi(L\times C)$, where $\phi: N\times C \to K'$   is the canonical homomorphism. 
 By our hypothesis, $L \times C$ is locally essential in $N \times C$. This implies that $c(K') = \phi(L\times C)$ is  locally essential in $K'$.
 Indeed, $N_* := \phi(N\times \{0\})$ is a $\delta$-subgroup of $K'$, so also a co-NSS subgroup since $K'$ is finite-dimensional.  Since 
 $\phi$ restricted to $N\times \{0\}$ is injective, so $\phi(L\times \{0\}) = \phi(L\times H) = c(K')\cap N_*$ is  locally essential in $N_* $, hence $c(K')$ is  locally essential in $K'$, in view of  $\dim C < \infty$ and Claim \ref{NewClaim*}. 
  \end{proof}
  
Item (c) describes the groups $K\in \KK$ as obtained by ``gluing'' a connected compact Abelian group $C$ with $\dim C< \infty$ and a totally disconnected compact Abelian group $N$ along a $\delta$-subgroup $H$ of $C$  and an essential closed subgroup $L$ of $N$ (via an isomorphism $j: L \to H$). The intersection $I$ of the isomorphic
  images of $C$ and $N$ in $K$ is isomorphic to $L \cong H$. Then $K/I \cong C/H \times N/L\not \in \KK$
  unless $N/L$ is finite, i.e., $L$ is open in $N$ (see Remark \ref{New:Remark}).  This shows that $\KK$ fails to be stable under taking quotients. 
  
The completions $K$ of the metrizable critical locally minimal precompact Abelian groups we build in the sequel belong to $\KK$. On the other hand, we build also a plenty of $\omega$-bounded critical locally minimal Abelian groups $G$. Their completions $K$ cannot belong to $\KK$, since otherwise 
$G/c(K)$ (along with $K/c(K)$) would be metrizable by Lemma \ref{D:claim}, as well as  $\omega$-bounded, so 
$G/c(K)$ would be compact. This would yield that $G$ itself is compact, a contradiction. 

%Theorem \ref{Thm:New}
\begin{theorem}\label{Thm:New} Let $K\in \KK$  and $C:= c(K)$. Then %$K \in \CCC$ iff $\pi^*(K/C)$ is infinite.
the following are equivalent:
\begin{itemize}
    \item[(a)] $K\in \CCC$;
    \item[(b)] $\pi^*(K/C)$ is infinite;
    \item[(c)] $K/C$ is infinite.
\end{itemize}
 \end{theorem}

\begin{proof} By Lemma~\ref{D:claim}(a), $td_p(K/C)$ is finite for every prime $p$. Therefore, conditions (b) and (c) are equivalent.

The implication (a) $\Rightarrow$ (c) is trivial, since otherwise $C$ would be open in $K$, and hence every dense subgroup $G$ 
containing $C$ is open (hence, closed) as well, so $G=K$. As $c(G) = C$, the quotient $G/c(G)$ is finite and hence locally minimal. 
This contradicts the assumption that $K \in \CCC$.

To see (b) $\Rightarrow$ (a), we take a finitely generated dense subgroup $A$ of $K/C$, which exists by Lemma \ref{D:claim}. Then the subgroup $G:=q^{-1}(A)$ of $K$, where $q: K\to K/C$ is the quotient homomorphism, is dense in $K$ and contains $C$. As $C$ is locally essential in $K$, so is $G$, and therefore $G$ is locally minimal.
But $G/C=A$ is not locally essential in $K/C$ because as a finitely generated group, the subgroup $t(A)$ is finite, while every \nbd\ of $K/C$ contains infinitely many 
elements of $t(A)$ because $K/C$ is a product of infinitely many cyclic groups.
\end{proof}

We will build a finite-dimensional ($\omega$-bounded) critical locally minimal precompact Abelian group $G$ such that $c(G)$ can be any connected finite-dimensional  compact Abelian group that is not a hybrid torus. From a different point of view, the next proposition can be considered as a partial inverse of Theorem \ref{coro:tor_free**}(d).

\begin{proposition}\label{PropSuff} Let $C$ be a finite-dimensional compact connected Abelian group that is not a hybrid torus.  Then there exists a compact Abelian group $K$ with $c(K) \cong C$ (so $\dim K$ is finite) containing a dense critical locally minimal subgroup $G$. Moreover, the group $G$ can be chosen to be:
\begin{itemize}
   \item[(a)] metrizable; or
   \item[(b)] $\omega$-bounded and of weight $\kappa$, for an arbitrarily assigned uncountable cardinal $\kappa$. 
\end{itemize}
\end{proposition}

\begin{proof}  Since $n:= \dim C< \infty$, the group $C$ is metrizable, so $C$ has a $\delta$-subgroup  $H$ such that  $C/H \cong \T^n$ and $P := \pi^*(H)$ is infinite, according to Fact \ref{Fact:Delta1} and our hypothesis. 

(a) By Fact~\ref{Fact1}(h), we have $H \cong \prod_{p \in P} H_p$, where each $H_p \cong \Z_p^{k_p} \times F_p$ with $F_p$ a finite $p$-group and $0 \leq k_p \leq n$.
%%%%%%%%%%%%%%%%%%%%%%%%% for the record %%%%%%%%%%%%%%%%%%%
%\footnote{He we "squeeze" this $H$ by taking a small subgroup $H'$ to use it for gluing with the special (similarly small) 
%$N$ built  to accomodate this (small) $H'$.
%We can leave this construction as it is, of course, to avoid problems with (b). 
%It has the advantage to give a monothetic $N/L$, allowing to have $G$ with ridiculously non-locally minimal (cyclic)
%$G/c(G)$. I gave an extended comment in the new Remark \ref{New:Remark} to outline the difference. }
%%%%%%%%%%%%%%%%%%%%%%%%%%%%%%%%%%%%%%%%%%%%%
Thus, for each $p \in P$, the group $H_p$ admits a topological direct decomposition of the form  $H_p = H_p' \oplus H_p''$, where $H_p'$ is topologically isomorphic either to $\Z_p$ or to a finite cyclic $p$-group. We then write $H$ as $H'\oplus H''$ with $H'\cong \prod_{p\in P}H_p'$ and $H''\cong \prod_{p\in P}H_p''$.

Now define the group $N := \prod_{p \in P} N_p,$ where each $N_p$ is isomorphic to $\Z_p$ if $H_p' \cong \Z_p$, or $N_p\cong \Z(p^{m+1})$, 
if $H_p'$ is cyclic of order $p^m$.   Then $N$ is an infinite monothetic group. Let $g \in N$ be an element such that $\hull{g}$ is dense in $N$.
Set $L := \prod_{p \in P} p N_p$, and fix a topological isomorphism $j: L \to H'$.

The graph $\Gamma_j = \{(x, j(x)) : x \in L\}$ of $j$ is a closed subgroup of $N \times C$ since 
$j$ is continuous. Let $K = (N \times C)/\Gamma_j$, with $\phi: N \times C \to K$ the canonical quotient homomorphism.  
Then $\phi(\{0\} \times C) = \phi(L \times C)$ coincides with $c(K)$. Clearly, $K$ is metrizable, as $C$ and $N$ are metrizable.

In order to show that $K \in \KK$ we check first that $N^* := \phi(N \times H'') \cong N \times H''$ is a $\delta$-subgroup of $K$.  Indeed, $N^*$ is totally disconnected and
$$
K/N^* \cong (N \times C)/\phi^{-1}(N^*) = (N \times C)/(N \times H'' + \Gamma_j) = (N \times C)/(N \times (H'' + H')) \cong C/H
$$
is a torus.  As $L^*:= N^*  \cap c(K) 
%\phi(N \times H'') \cap \phi(L \times C) 
= \phi(L \times H'')$ is essential
in $N^*$, 
%\NB\footnote{was Claim~\ref{NewClaim}}
Claim~\ref{NewClaim*} implies that $ c(K)$ is  locally essential in $K$, so $K \in \KK$. Moreover, $\pi^*(N^*/L^* )$ is infinite, as 
$N^*/L^*\cong   \prod_{p\in P} \Z(p)$. Hence,   $K \in \CCC$, by Theorem \ref{Thm:New}. 
%%%%%%%%%%%%%%%%%%%%%%%%%%%%%%%%%%%%%%%%%%
%Also, since $\ker \phi = \Gamma_j \leq L \times C$, we obtain
%\begin{equation} \label{Equ}
%N^* \cap c(K) = \phi(N \times H'') \cap \phi(L \times C) = \phi(L \times H'').
%\end{equation}
%%%%%%%%%%%%%%%%%%%%%%%

\bigskip

(b) In order to produce an $\omega$-bounded group $G$ as above, we need to modify the construction as follows.

We retain the notation from part~(a); in particular, the groups $N$, $L$, $H = H' \oplus H''$, the map $j$, the homomorphism $\phi$, and the subgroup $N^*$ 
of $K$ are as defined therein. Note that $N^*$ decomposes as $N^* = N^*_1 \oplus N^*_2$ with 
$$N^*_1 = \phi(N\times \{0\}) \cong \prod_{p\in P} N_p, \quad N^*_2 = \phi(\{0\}\times H'')\cong H''.$$
In the sequel, we identify $N^*_1$ with $\prod_{p\in P} N_p$.
We denote by $K_0$ the compact Abelian group constructed in part~(a); that is, $K_0 = (N \times C)/\Gamma_j.$

\medskip 

1. [{\sl Definition of $K$}] Let $B = \prod_{p\in P} B_p$, where $B_p = N_p^\kappa$,
% and $\kappa > \omega$, 
and set $K := K_0 \times B$ endowed with the Tychonoff product topology. Then $c(K) = c(K_0) \cong C$.

By standard isomorphism theorems for topological groups, we obtain:
\begin{equation}\label{Equ1}
K/c(K) = (K_0 \times B)/c(K_0) \cong K_0/c(K_0) \times B \cong B \times\prod_{p\in P} \Z(p).
\end{equation}

Let $N_* := N^* \oplus B$, and let $\varphi_1: N_* \to N^*$ and $\varphi_2: N_* \to B$ be the natural projections. %As in part~(a), we denote by $q: N_* \to N_*/\phi(L \times H'')$ the canonical quotient map.
As $N_*$ is a $\delta$-subgroup of $K_0$ and $B$ is totally disconnected, $N_*$ $N_*$ is a $\delta$-subgroup of $K$, by Fact \ref{Fact1}(e). 
%1.  Since $N_*$ is clearly totally disconnected and $K/N_* = (K_0\oplus B)/(N^*\oplus B) \cong K_0/N^*$ is a torus, we conclude that $N_*$ is a $\delta$-subgroup of $K$.
\medskip 

2. [{\sl Definition of $G$}]: For every $p \in \P$ let $S_p$ denote the $\Sigma$-product in $N_p^\kappa = B_p$
and $E_p = S_p + pN_p^\kappa$. Then $E := \prod_{p\in P} E_p$ is a dense essential $\omega$-bounded subgroup of $B$.  

Furthermore, for each $p \in P$, the subgroup $td_p(N^*_1)\oplus td_p(B)= N_p \oplus (N_p)^\kappa$ of $K$ is isomorphic to
$(N_p)^{1+\kappa}$.  Let $\Delta_p \cong N_p$ be the diagonal subgroup of this Cartesian power. Set $D := \prod_{p \in P} \Delta_p$. Then $\varphi_1(D) = N^*$.

The subgroup $G_1 := \phi(L \times H'') \oplus E + D$ of $ N_*$ is dense, $\omega$-bounded, and essential in $N_*$. Finally, define
\begin{equation}\label{Eq:June4*}
G := G_1 + c(K)=E + D + c(K).
\end{equation}
Clearly, $c(G) = c(K) = c(K_0)$. By 
%\NB\footnote{was Claim~\ref{NewClaim}}
Claim~\ref{NewClaim*} and the fact that $N_*$ is a $\delta$-subgroup of $K$, it follows that $G$ is locally essential in $K$, hence locally minimal.
 
 \medskip 

3. Let $q: K \to K/c(K)$ be the quotient homomorphism. We now prove that 
%\NB\footnote{was: 
%\begin{equation}\label{Eq:May7}
%q(G) \cap q(K) = \{0\}.
%\end{equation}
%First note that  
%\begin{equation}\label{Eq:May7'}
%q(G) \cap q(K)= q(G\cap K),
%\end{equation}
% since $q(g) = q(z)$, for $g\in G, z \in K$, gives $z\in q^{-1}(q(G)) = G$ (since $\ker q \leq G$), so $z\in G\cap K$. This proves the inclusion  ``$\leq$'' in \eqref{Eq:May7'}, the other inclusion is trivial. }
\begin{equation}\label{Eq:May7}
q(G) \cap q(K_0) = \{0\}.
\end{equation}
First note that  
\begin{equation}\label{Eq:May7'}
q(G) \cap q(K_0)= q(G\cap K_0),
\end{equation}
 since $q(g) = q(z)$, for $g\in G, z \in K_0$, gives $z\in q^{-1}(q(G)) = G$ (since $\ker q \leq G$), so $z\in G\cap K_0$. This proves the inclusion  ``$\leq$'' in \eqref{Eq:May7'}, the other inclusion is trivial. 

In view of \eqref{Eq:May7'}, to prove \eqref{Eq:May7} it suffices to show that 
\begin{equation}\label{Eq:May7*}
K_0 \cap G = c(K) = c(K_0).
\end{equation}
%From 
%$$
%G_1 = \phi(L \times H'') \oplus E + D \quad \text{and} \quad \phi(L \times H'') \leq c(G),
%$$
%we obtain
%$$
%G = G_1 + c(G) = E + D + c(G).
%$$
As  $c(K) \leq K_0$, the modular law and \eqref{Eq:June4*} give, 
$G \cap K_0 = (E + D + c(G)) \cap K_0 = (E + D) \cap K_0 + c(K)$. Hence, in view of 
$c(K) \leq K_0 \cap G$, to check \eqref{Eq:May7*} it suffices to prove that $(E + D) \cap K_0 \leq c(K)$.

Take an element $g = e + d \in (E + D) \cap K_0$, where $e \in E \leq B$ and $d \in D$. Write $d = d' + d''$, where $d' = \varphi_1(d) \in K_0$ and $d'' = \varphi_2(d) \in B$. Then
$$
g = d' + (d'' + e).
$$
Since $g \in K_0$ and $K_0 \oplus B$ is a direct sum, it follows that $g = d'$ and $d'' + e = 0$, so $d'' = -e \in E = \prod_{p \in P} E_p$.

Fix $p \in P$. By definition, $E_p$ consists of functions $f: \kappa \to N_p$ such that $f(\alpha) \in p N_p$ for all but countably many $\alpha < \kappa$. On the other hand, $d''(p)$ lies in the diagonal subgroup $\varphi_2(\Delta_p)$ of $B_p = N_p^\kappa$, so $d''(p)$ is a constant function. Therefore, $d''(p) \in p B_p$. Hence, $d'' \in \prod_{p \in P} p B_p.$

Since $d = d' + d'' \in D = \prod_{p \in P} \Delta_p$ and $\Delta_p$ is the diagonal subgroup of $td_p(N_1^*) \oplus B_p \cong N_p^{1+\kappa}$, we conclude that
$$
g=d' \in \prod_{p \in P} p N_p = \phi(L \times \{0\}) \leq c(K),
$$
as claimed.

\medskip 

4. Finally, we prove that $G$ is critical locally minimal. In 2. we proved that $G$ is locally minimal, hence it remains to show that $G/c(K) = q(G)$ is not locally minimal. 
By Theorem~\ref{Crit}, this amounts to proving that $G/c(G)$ is not locally essential in $K/c(K)$. Since $K/c(K)$ is totally disconnected, it will suffice, 
in view of Theorem \ref{claim4}(c), to show that $\pi_{(q(G), K/c(K))}$ is infinite, i.e., $q(G)$ does not contain 
$$
(K/c(K))[p] = (N/L \times B)[p] = (N/L)[p] \times B[p]
$$
for infinitely many primes $p$.
Since
$$
G/c(G) = q(K_0) \oplus q(B) \quad \text{and} \quad q(K_0) = K_0 / c(K) = K_0 / c(K_0) \cong \prod_{p \in P} \Z(p),
$$
one can conclude with \eqref{Eq:May7} that $\pi_{(q(G), K/c(K))}=P$, i.e., $q(G)$ fails to contain $(K/c(K))[p]$ {\em for  every} prime $p\in P$.  Therefore, $q(G)$ is not locally essential in $K/c(K)$ and hence not locally minimal.
\end{proof}

In the final part of the proof of (a) we showed that $ c(K)$ is  locally essential in $K$, i.e., $K \in \KK$. Nevertheless, this  construction of $K\in \KK$ (starting from $C$) differs from the construction of $K\in \KK$ in Theorem \ref{inKK}(c), which starts from $C$ and some {\em $\delta$-subgroup} $H$ of $C$. Here we glue $C$ and the totally disconnected group $N$ along a smaller subgroup $H'$ of $C$ that need not be a  $\delta$-subgroup of $C$ and an essential subgroup  $L$ of $N$ isomorphic to this $H'$. 
From the point of view of the quotient $N/L\cong K/C$, the group $K$ built  here is as small as possible (modulo finite).

\medskip

Since $\pi(C) = \pi^*(C) = \P$ is infinite for connected torsion-free compact Abelian groups $C$ (so they are not hybrid tori), we obtain: 

\begin{corollary}\label{exa:Thm:qt*}
For every connected torsion-free finite-dimensional compact Abelian group $C$ there exists a compact Abelian group $K\in \CCC$ with $c(K) \cong C$.
\end{corollary}

The groups $C$ described by the above properties are exactly the powers $\K^n$, $n\in \N$. 
 
\begin{proof}[{\bf Proof of Theorem \ref{IntroB}}] We have to prove that if $C$ is a finite-dimensional connected compact Abelian group, then the following are equivalent:
\begin{itemize}
\item[(a)] $C$ is a not a hybrid torus;
\item[(b)] there is a compact Abelian group $K\in \CCC$ with $c(K)\cong C.$ Moreover, the dense critical locally minimal subgroup $G$ of $K$ witnessing $K\in \CCC$
$G$ can be chosen to be either metrizable or $\omega$-bounded and of weight $\kappa$, for an arbitrarily assigned uncountable cardinal $\kappa$. 

\end{itemize}

 It is enough to combine Proposition \ref{PropSuff} and Theorem \ref{coro:tor_free**}(d). \end{proof}

%\begin{theorem}\label{MainThm} \end{theorem}

 For every infinite cardinal $\alpha < \kappa$ we can replace in (b) $\omega$-bounded by $\alpha$-bounded, since our construction produces $\alpha$-bounded groups by replacing the current $\Sigma$-product $S$ by the $\alpha$-product defined similarly.

Item (b) cannot be pushed to claim that {\em every } compact Abelian group $K$ with $c(K)\cong C$ belongs to $\CCC$. In other words, Theorem \ref{IntroB}
%\ref{MainThm} 
characterizes only the connected component $c(K)$ of a finite-dimensional group $K$ in $\CCC$, but not the 
 whole group $K$.  We give now two examples to this effect.

\begin{example} A group $K\in \CCC$ cannot be connected, hence this imposes a first restriction (i.e., beyond 
$c(K)\cong C$ one has to impose $K \ne c(K)$ regardless of whether $C$ is or is not a hybrid torus). 
 %$C = \K$ is not a hybrid torus, but $K=C$ with $c(K) = C$ has no dense critical locally minimal subgroups

(a) If $K$ is a torsion-free compact Abelian group that is neither connected nor totally disconnected%\NB \footnote{dp: This result holds for any 1-dimensional connected compact Abelian group $K$ indeed, since for any dense subgroup $G$ of $K$, as $\dim K=1$, either %$G=K$ or $c(G)=\{0\}$.\\
%\NB I didn't quite understand what you mean, but maybe my interpretation goes in the right direction. }
, then $c(K)$ is not a 
hybrid torus, yet $K\not \in \CCC$, since every dense locally minimal subgroup of $K$ is also minimal, so cannot be critical locally minimal, by Corollary \ref{Coro:MinG/c(G)}. 
Another argument is given in the item below. 
%since $c(K)$ splits topologicaly (or, ). 

(b) If $C$ is not a hybrid torus, but $K=C \times D$ for some totally disconnected compact Abelian group $D$ (so that $c(K) =C$), then 
$K \not \in \CCC$ by Theorem \ref{coro:tor_free**}.
\end{example}

\begin{remark}\label{Homotop} Theorem \ref{IntroB} shows, among others, that there is a proper class of pairwise non-homeomorphic groups $K\in \CCC$ with a fixed connected component. Indeed, for every uncountable cardinal $\kappa$ and for every connected compact Abelian group $C$ that is not a hybrid torus
 there exists a compact Abelian group $K_{C,\kappa}\in \CCC$ with $w(K_{C,\kappa}) = \kappa$ and $c(K_{C,\kappa}) \cong C$
containing a dense $\omega$-bounded critical locally minimal subgroup $G_\kappa$ (it can be chosen to be even $\alpha$-bounded for every infinite $\alpha< \kappa$). 
Due to the equality $w(K_{C,\kappa}) = \kappa$, for every fixed $C$ the groups $K_{C,\kappa}\in \CCC$ are pairwise non-homeomorphic and obviously form a proper class. 

Now fix the uncountable cardinal $\kappa$ and vary $C$. Then one can arrange to have for $\mathfrak c$ many of the groups $K_{C,\kappa}$ 
the connected components $C$ pairwise non-homotopically equivalent. This follows from the fact that if two connected compact Abelian groups 
$C, C_1$ are homotopically equivalent, then they are also topologically isomorphic, since 
for a connected compact Abelian group $C$, the abstract group $\widehat C$  is isomorphic to the first \v Cech cohomology group of $C$.
To this we can add that for every infinite set of primes $P$ there is a family of pairwise non-topologically isomorphic one-dimensional compact and connected Abelian groups $C$ with $\pi(C) = P$ as those used in Theorem \ref{IntroB}.
%\ref{MainThm}. 
\end{remark}

\section{Final comments and open questions}\label{Sec:Final}

In Corollary \ref{Coro2:May27}, we showed that every complete locally minimal Abelian group $G$ contains a compact $G_\delta$-subgroup $K$.
As a consequence, the quotient group $G/K$ has countable pseudocharacter; that is, the singleton $\{0_{G/K}\}$ is a $G_\delta$-set of $G/K$. 
On the other hand, a topological group $G$ is called \emph{almost metrizable}~\cite{Pas}, or equivalently \emph{feathered}~\cite[Sec.~4.3]{AT}, if there exists a compact subgroup $K$ of $G$ such that $G/K$ is metrizable—or, equivalently, first-countable, by~\cite[Lemma~4.3.19]{AT}. This naturally leads to the following question:

\begin{question}\label{QuesX1}
Is every complete locally minimal Abelian group $G$ feathered?
\end{question}

The answer to this question would be \emph{affirmative}, if for the specific complete locally minimal Abelian group $G$
and its compact $G_\delta$-subgroup $K$ of $G$ the quotient $G/K$ is locally minimal, since locally minimal Abelian groups of countable pseudocharacter are metrizable. This triggers another question closely related to the second principal line of this paper (namely, the local minimality of the quotient $G/c(G)$ of locally minimal groups $G$ 
when the subgroup $c(G)$ is compact): 

\begin{question}\label{QuesX2} If $G$ is a complete locally minimal Abelian group and $K$ is a compact ($G_\delta$-subgroup) subgroup of $G$, is then 
the quotient $G/K$ locally minimal?
\end{question}

The above observation shows that a positive answer to Question \ref{QuesX2} yields a positive answer to Question \ref{QuesX1}.   

Question \ref{QuesX1} is equivalent to asking whether every complete locally minimal Abelian group is Čech complete, since by a theorem of Choban, a topological group is Čech complete if and only if it is both feathered and complete~\cite[4.3.16]{AT}.

\medskip

 The following general fact ensures ``automatic" minimality of dense subgroups of nilpotent class 2 minimal groups. 

\begin{claim}\label{Easy:Claim}  The center $Z(G)$ of every nilpotent group $G$ of class 2 is essential. Hence, every subgroup of $G$ containing $Z(G)$ is essential.  
\end{claim}
\begin{proof} Let $N$ be a non-trivial normal subgroup of $G$ and $e \ne u \in N$. If $u\in  Z(G)$ we are done. If $u\not \in Z(G)$, 
then there exists $v\in G$ such that $[u,v]\ne e$.  Since $vu^{-1}v^{-1}\in N$, we conclude that $e\ne [u,v]=uvu^{-1}v^{-1}\in N$. On the other hand, since $G$ is nilpotent of class two, $[u,v]\in [G, G]\subseteq Z(G)$. This proves that $N \cap Z(G)\ne \{e\}$.
\end{proof}

The next example shows that Proposition \ref{tot:disco+cc+loc:min} cannot be extended to the non-Abelian case even if one strengthens sequential completeness to $\omega$-boundedness and local minimality to  minimality:
% (we keep on using additive notation for the sake of convenenience\NB\footnote{I am wondering if it makes sense to explicitly sgive this example without explaining why the set $p^kA$ is a subgroup, etc., etc.}):  

\begin{example}\label{New:Example} For every prime  $p$ there exists a pro-$p$-finite nilpotent group $A_p$ of class 2 and 
an $\omega$-bounded dense minimal subgroup $G$ of $A_p$, such that $A_p^{p^k} \not \leq G$ for every $k \in \N$.
Indeed, let $R$ be the compact ring $\Z_p^\mathfrak c$ and let $S$ be the $\Sigma$-product in $R$. 
Let $A_p=
\left(\begin{matrix}
  1 &  R & R \cr
\,0 & 1 & R \cr
\,0 & 0 & 1 \cr
                 \end{matrix}
                 \right)
$ be the Heisenberg group over $R$. Then its dense 
$\omega$-bounded subgroup $G=\left(\begin{matrix}
  1 & S & R \cr
\,0 & 1 & S \cr
\,0 & 0 & 1 \cr
                 \end{matrix}
                 \right)$ has the above mentioned properties, as $G$ is minimal (by the above claim, since $G\geq Z(A_p)$) and  
                 $$A_p^{p^k}=
\left(\begin{matrix}
  1 &  {p^k}R & {p^k}R \cr
\,0 & 1 & {p^k} R \cr
\,0 & 0 & 1 \cr
                 \end{matrix}
                 \right) \not \leq G.$$ 
 \end{example}

The next extended example shows that Corollary \ref{Coro:MinG/c(G)} fails in many ways for non-Abelian groups (even when they are nilpotent class two), so it motivates our choice to concentrate only on the Abelian case in this paper.

\begin{example} Let $X$ be a LCA group. Denote by ${\mathcal H}(X)$ the Heisenberg group introduced by Megrelishvili in \cite{Meg} (see also  \cite[\S 5]{DMeg}) defined as follows. The topological group $\mathcal H(X)$ has as a supporting topological space  the Cartesian product $\T \times X \times \widehat X$ equipped with the product topology, and group operation defined for 
\begin{equation}\label{(Last)}
u_1=(a_1,x_1,f_1), \hskip 0.4cm u_2=(a_2,x_2,f_2)
\end{equation}
in $\mathcal H(K)$ by 
$$
u_1 \cdot u_2 = (a_1+a_2+f_1(x_2), x_1+x_2, f_1 +f_2)
$$ 
In other words,  $\mathcal H(X)$ is the semidirect product  $(\T \times X ) \leftthreetimes \widehat X$ of the groups $\widehat X$ and  $\T \times X$, where the action of $\widehat X $ on $\T \times X$ is defined by $(f,(a,x))\mapsto (a+f(x),x)$ for $f\in \widehat X $ and $(a,x) \in \T \times X $. Then $\mathcal H(X)$ is locally compact and minimal \cite{Meg}. 

It is convenient to describe the group $\mathcal H(X)$ also in the matrix form
$$
\left(\begin{matrix}
  1 &  \widehat X & \T \cr
\,0 & 1 & X \cr
\,0 & 0 & 1 \cr
                 \end{matrix}
                 \right)
$$
so that the multiplication of two elements of $\mathcal H(X)$ is carried out precisely by the rows-by-columns rule for multiplication of $3 \times 3$ matrices, where the ``product'' of an element $f\in \widehat X$ and and element $x\in X$ has value $f(x) \in \T$. Elementary computations for the commutator $[u_1,u_2]$ of $u_1, u_2\in \mathcal H(X)$ as in (\ref{(Last)}) give $[u_1,u_2] = u_1u_2u_1^{-1}u_2^{-1}= (f_1(x_2)-f_2(x_1),0_X,0_{\widehat{X}})$ hence 
$$Z(\mathcal H(X))=\left(\begin{matrix}
   1 & 0 & \T \cr
 \,0 & 1 & 0 \cr
 \,0 & 0 & 1 \cr
                 \end{matrix} \right). $$
                 
The quotient $\mathcal H(X) /Z(\mathcal H(X))$ is isomorphic to the (Abelian) direct product $X \times  \widehat X$. Therefore,  $\mathcal H(X)$ 
is a nilpotent group of class 2. We identify $X$ with $\{0_\T \} \times X \times \{0_{ \widehat X} \}$,  $ \widehat X$ with $\{0_\T\} \times \{0_X\} \times  \widehat X$
and  $ \T$ with $\T \times \{0_X\} \times  \{0_{\widehat X}\}$. 

Now assume that $X$ is compact, so $ \widehat X$ is discrete. For a subgroup $H$ of $X$ we denote by $G_H$ the set 
$$ \left(\begin{matrix}  
  1 & \widehat X  & \T \cr
\,0 & 1 &  H  \cr
\,0 & 0 & 1 \cr
                 \end{matrix} \right) :=  \left\{\left(\begin{matrix}
  1 & x & t \cr
\,0 & 1 & h \cr
\,0 & 0 & 1 \cr
                 \end{matrix} \right)\in \mathcal H(X): h\in H, x\in \widehat{X}, t\in \T
                                  \right\}, $$
which is a subgroup of $\mathcal H(X)$. Since $\mathcal H(X)$ is homeomorphic as a topological space to the locally compact 
Cartesian product $\T \times X \times \widehat X$, 
$G_H$ is locally precompact. Moreover, $G_H$ is dense (resp., sequentially closed) in $\mathcal H(X)$ iff $H$ is dense (resp., sequentially closed) in $X$. In particular, $G_H$ is \sco \ iff $H$ is \sco. The inclusion $Z(\mathcal H(X))\subseteq G_H$ and Claim \ref{Easy:Claim} imply that $G_H$ is minimal whenever it is dense in $\mathcal H(X)$. 
%every non-trivial normal subgroup $N$ of $\mathcal H(X)$ non-trivially meets $Z(\mathcal H(X))\subseteq G_H$. 
\medskip 

 Clearly, $c(\mathcal H(X)) = \left\{\left(\begin{matrix}
  1 & 0 & \T \cr
\,0 & 1 & c(X) \cr
\,0 & 0 & 1 \cr
                 \end{matrix} \right)\right\}$, while 
$c(G_H) = \left\{\left(\begin{matrix}
  1 & 0 & \T \cr
\,0 & 1 & c(H) \cr
\,0 & 0 & 1 \cr
                 \end{matrix} \right)\right\}$, so $G_H/c(G_H) \cong 
                 H/c(H) \times \widehat X$.
                 (The topological group isomorphism ${\phantom{I^{I^{|^{|^I}}}} \!\!\!\!\!\!\!\!\!\!\!}  \mathcal H(X) /Z(\mathcal H(X)) \cong X \times  \widehat X$ 
does not yields that the {\em subset} $\{0\} \times X \times  \widehat X$ of $\mathcal H(X)$
 is a {\em subgroup} of the group $\mathcal H(X)$.) 

The minimal group $G_H$ will be used in (a)--(c) below to see that both Theorem \ref{Thm:DTk} and Theorem \ref{IntroA} fail for non-Abelian groups even when they are nilpotent of class 2 (i.e., as close as possible to Abelian). To this end we need to carefully chose $H$ to be a sequentially complete dense subgroup of some compact Abelian group $X$.

(a) Let $X=b\Z$ be the Bohr compactification of the discrete group $\Z$. It is well known that $\Z$, equipped with the induced topology
is $\Z^\#$. Then $H:= \Z^\#$ is  sequentially complete (see Example \ref{Exa:Bohr}) and $c(H) =  \{0\}$. Therefore,  $c(G_H) \cong \T$ is metrizable, whereas $c(\mathcal H(X)) \cong c(X) \times \T$ has weight $\mathfrak c > w(\T) = w(c(G_H))$,  since $c(X) \cong \K^\mathfrak c$. Therefore both Theorem \ref{Thm:DTk} and Theorem \ref{IntroA}
   %\ref{loc:min>comp*}(a) 
  fail for this minimal \sco \ and locally precompact group $G_H$.
  
 The gap between  $w(c(G_H))$ and $ w(c( \widetilde G_H))$ can be made arbitrarily large by choosing $H =(\bigoplus_\kappa \Q)^\#$ and 
  $X = b H \cong \K^{2^\kappa}$. So $X=c(X)$ and $w(c(X)) = 2^\kappa$, while $c(H) = \{0\}$ since the $\Q/\Z$-valued characters of $H$ separate the points of 
  $H$, so there is a continuous injective homomorphism $H\to (\Q/\Z)^H$, where $(\Q/\Z)^H$ carries the (zero-dimensional) product topology
  and $\Q/\Z$ the topology induced by $\T$. 

 (b) Now we arrange to have $G_H$ such that $w(c(G_H)) = w(c( \widetilde G_H))$ is not Ulam measurable, yet $c(G_H) \ne c( \widetilde G_H)$ (witnessing failure of Theorem \ref{Thm:DTk} and Theorem \ref{IntroA} (b), while Theorem \ref{IntroA} (a)  holds true). Take the compact Abelian group $Y = b\Z$ (as in (a)), $X = Y \times Y$ and the dense subgroup $H = Y \times \Z^\#$ of $X$. Then $H$ is \sco, so $G_H$ is \sco \ as well.  Moreover, $c(H) \cong c(Y) \times \T \cong \T \times \K^\mathfrak c$, while $c(\widetilde G_H) \cong X \times \T$. Hence, $c(G_H) \ne c( \widetilde G_H)$, yet $w(c(G_H)) = w(c(\widetilde G_H))=\mathfrak c$. Therefore, both Theorem \ref{Thm:DTk} and Theorem \ref{IntroA} (b) fail is $\mathfrak c$ is not Ulam-measurable, while Theorem \ref{IntroA}(a) holds true for this group $G_H$. 
  
 %Then $X$ is monothetic, so has a dense  cyclic subgroup $H$, so $c(H) = \{0\}$ (any hereditarily disconnected dense subgroup $H$ of $X$ will do).  non-measurable size and a proper dense \sco \ subgroup $H$ of $X$. Then $G_H$ will be a \sco \ minimal group with completion $\mathcal H(X)$ and  non-measurable $w(c(G_H))$, yet $G_H$ fails to contain $c(\mathcal H(X)$. 
  
  (c) Finally, we arrange to have $G_H$ with $c(G_H) = c( \widetilde G_H)$, but $G_H/c(G_H)$ is not even locally minimal (witnessing failure of Theorem \ref{Thm:DTk} 
  and Corollary \ref{Coro:MinG/c(G)},
  while both items (a) and (b) of Theorem \ref{IntroA}
  %\ref{loc:min>comp*} 
  hold true). This time we pick the compact Abelian group $X\in \CCC$  and a dense $\omega$-bounded subgroup $H$ of $X$ which satisfies $c(H) = c(X)$ and $H/c(H)$ is not locally minimal as built in Proposition \ref{PropSuff} (a much easier example is obtained by taking $X = \T \times \Z_p^\mathfrak c$, and $H = \T \times H_1$, where $H_1$ is the $\Sigma$-product in $\Z_p^\mathfrak c$, so that $H/c(H)\cong H_1$ is not locally minimal since it is not minimal, according to Corollary \ref{New:Corollary}). Then  $c(G_H) = c( \widetilde G_H)$, yet $G_H/c(G_H) \cong H/c(H) \times \widehat X$ is not even {\em locally minimal}, as the non-locally minimal group $H/c(H)$ is isomorphic to an open subgroup of $G_H/c(G_H)$.  
\end{example}

The minimal group $G_H$ we built in the above example is precompact if and only if it is finite. 
Hence, to show that 
%both Theorem \ref{loc:min>comp*}   and 
Corollary \ref{Coro:Oct24} strongly fails in the non-Abelian case we need another example:  

\begin{example}
Let $L$ be a compact connected simple Lie group (e.g., $L=\mathrm{SO}_3(\R)$ will do).  Then in the compact connected group $K=L^{\omega_1}$ the $\Sigma$-product $G$ is an $\omega$-bounded (hence, \sco) totally minimal connected non-compact group. (The total minimality of $G$ is due to Theorem \ref{Crit}(ii) and the fact that the only closed normal subgroups $N$ of $K$ are the products $N= L^I$, for some $I \subseteq \omega_1$, so $N\cap G$ is always dense in $N$.)
\end{example}

The counterexamples to  Theorem \ref{IntroA}
%\ref{loc:min>comp*} 
and Corollary \ref{Coro:Oct24} given above have separately the properties of being nilpotent and precompact. 
So the question arises of whether one can produce a counterexample that is both  nilpotent and precompact. Let us see that this is not possible
as far as item (a) of the theorem is concerned. 

For the proof we need the following (probably known) fact, we include a proof here for convenience of the reader. 

\begin{fact}\label{central} Let $G$ be a nilpotent compact group. Then $c(G) \leq Z(G)$.
\end{fact}

\begin{proof} Assume first that $G$ is a Lie group. Then, by a well-known theorem, $G$ admits a finite subgroup $E$ such that $G = c(G)E$.
Let $A = c(G)$; then $A$ is nilpotent and hence Abelian, so $A \cong \T^m$ for some $m\in \N$. 
Choose a prime $p > |E|$.
Then the finite $p$-torsion subgroup $A[p]$ of $A$ is a characteristic subgroup of the normal subgroup $A$
of $G$. Hence, $A[p]$ is a  normal subgroup of $G$.
Since every finite nilpotent group is isomorphic to the direct product of its Sylow subgroups and $p$ is co-prime to $ |E|$, 
$A[p]$ is centralized by $E$.  Hence $A[p] \subseteq Z(G)$.

As the collection of all such $A[p]$ generates a dense subgroup of $A$, it follows that $A \subseteq Z(G)$.
Therefore, $c(G) \leq Z(G)$.

In the general case fix an arbitrary closed normal subgroup $N$ of $G$ such that $G/N$ is a Lie group, and denote by $\pi: G \to G/N$ the natural quotient homomorphism.
Then for any $c \in c(G)$ and $g \in G$, the first part of the proof implies that $[\pi(c), \pi(g)] = e$ in $G/N$, i.e., $[c,g] = cgc^{-1}g^{-1} \in N$.

Since the intersection of all such $N$ is trivial, it follows that $[c,g] = e$ for all $g \in G$, and thus $c(G) \leq Z(G)$.
\end{proof}

We do not know the answer of the following 

\begin{question}\label{LastQues}
Does Corollary \ref{Coro:Oct23} remain true in the nilpotent case, with ``locally precompact" replaced by ``precompact"?
\end{question}

 A positive answer to this question means that the triviality of $c(G)$ implies that also $c(\widetilde{G})$ is trivial
for a precompact, sequentially complete and locally minimal nilpotent group. This is precisely item (a) of Theorem \ref{IntroA} for 
{\em hereditary disconnected} precompact, sequentially complete and locally minimal nilpotent groups. 

Let us see now that Theorem \ref{IntroA} (a) holds for any precompact, sequentially complete and locally minimal nilpotent group $G$  with $c(G)\ne \{0\}$, i.e., the only case remaining open from Theorem \ref{IntroA} (a) is that described in Question \ref{LastQues}.  Indeed, $c(G)\ne \{0\}$ implies that $w(c(G))\geq \omega$, hence  $w(c(G))< w(c(\widetilde{G}))$ may occur only when  $w(c(\widetilde{G}))>\omega$. Therefore, it is enough to consider the case when $\dim \widetilde{G}$ is infinite.
Since the compact group $\widetilde{G}$ is nilpotent, $c(\widetilde{G})$ is central in $\widetilde{G}$, by Fact~\ref{central}.
Hence, every subgroup of $c(\widetilde{G})$ is normal in $\widetilde{G}$. In particular, $G \cap c(\widetilde{G})$ is a locally essential subgroup of $c(\widetilde{G})$.
%and we may assume without loss of generality, that $c(\widetilde{G})$ is not metrizable.As mentioned above, $c(\widetilde{G})$ is central in $\widetilde{G}$, so every subgroup of $c(\widetilde{G})$ is normal in $\widetilde{G}$. In particular, $G \cap c(\widetilde{G})$ is a locally essential subgroup of $c(\widetilde{G})$.
By Remark \ref{Laast:Rem}, $G \cap N$ is essential in some co-NSS subgroup $N$ of $c(\widetilde{G})$.
Therefore, arguing  in the line of the proof of Theorem \ref{IntroA} (a$'$), we conclude that $w(c(G)) = w(c(N)) = w(c(\widetilde{G}))$. 
%If $G$ is connected, then it is also Abelian, so $G=\widetilde{G}$ is compact. 

In spite of the (partially) positive solution towards obtaining a countrepart of Theorem \ref{IntroA} (a) for nilpotent groups, 
the answer to the following question regarding a countrepart of Theorem \ref{IntroA}(b)  remains unknown in the general case: 

\begin{question} Is Theorem \ref{IntroA} (b) valid for precompact, sequentially complete and locally minimal nilpotent groups? 
\end{question}

 Following the line of the proof of \cite[Theorem 3.2]{DTk} one can give a positive answer to this question if $G$ is nilpotent of class 2 and 
 precompactness is replaced by the stronger property of pseudocompactness. However, that proof cannot work in the precompact case since it makes essential use of Proposition \ref{tot:disco+cc+loc:min} (in the version
of minimality), but as Example \ref{New:Example} shows, the counterpart of that proposition fails even for nilpotent groups of class 2. 

\begin{question}\label{LAST:ques}  Does Theorem \ref{IntroA} remain true without the assumption of local precompactness? 
\end{question}

Of course, we intend item (b) of Theorem \ref{IntroA}  without its final part ``$c(G)$ is locally compact".

\bigskip

Dikran Dikranjan

Dipartimento di Matematica e Informatica, Universit\`{a} di Udine, 33100 Udine, ITALY

e-mail: dikran.dikranjan@uniud.it

\bigskip

%Wei He
%

Wei He

Institute of Mathematics, Nanjing Normal University, Nanjing 210046, CHINA

e-mail:  weihe@njnu.edu.cn

\bigskip

Dekui Peng

Institute of Mathematics, Nanjing Normal University, Nanjing 210046, CHINA

e-mail:  pengdk10@lzu.edu.cn
\bigskip

%WENFEI XI
%
%School of Applied Mathematics, Nanjing University of Finance \& Economics, Nanjing 210046, CHINA
%
%e-mail: xiwenfei0418@Outlook.com

\end{document}